\documentclass[a4paper,11pt]{amsart}
\usepackage{cite}
\usepackage{amsmath, amsthm}
\usepackage{textcomp}
\usepackage{amssymb}
\usepackage[all]{xy}
\usepackage{enumerate}
\usepackage{fancyhdr}
\usepackage{lscape}
\usepackage{url}
\usepackage{setspace}
\usepackage{longtable}

\theoremstyle{plain}
\newtheorem{theorem}{Theorem}[section]
\newtheorem{lm}[theorem]{Lemma}
\newtheorem{prop}[theorem]{Proposition}
\newtheorem{cor}[theorem]{Corollary}
\newtheorem{rmk}[theorem]{Remark}

\newtheorem{ex}[theorem]{Example}

\newtheorem{defi}[theorem]{Definition}

\DeclareMathOperator{\Pic}{Pic}
\DeclareMathOperator{\Int}{Int}

\DeclareMathOperator{\Sing}{Sing}
\DeclareMathOperator{\Proj}{Proj}

\DeclareMathOperator{\Div}{Div}

\DeclareMathOperator{\Cox}{Cox}

\DeclareMathOperator{\N}{N}

\DeclareMathOperator{\Exc}{Exc}
\DeclareMathOperator{\Sl}{SL}
\DeclareMathOperator{\MMob}{Mob}
\DeclareMathOperator{\Bs}{Bs}

\DeclareMathOperator{\NNE}{\overline{NE}}

\newcommand{\F}{\mathcal{F}}
\renewcommand{\O}{\mathcal{O}}
\newcommand{\U}{\mathcal{U}}
\newcommand{\C}{\mathbb{C}}
\newcommand{\Z}{\mathbb{Z}}
\newcommand{\Q}{\mathbb{Q}}

\newcommand{\R}{\mathbb{R}}

\newcommand{\III}{{\rm III}}
\newcommand{\IV}{{\rm IV}}

\renewcommand{\P}{\mathbb{P}}
\renewcommand{\H}{\text{H}}
\renewcommand{\N}{\text{N}}
\newcommand\qt{{\slash\kern-0.65ex\slash}}

\title{On non-rigid del Pezzo fibrations of low degree}
\date{}
\author{Hamid Ahmadinezhad}

\begin{document}

\maketitle

\begin{abstract}
We consider $\P(1,1,1,2)$ bundles over $\P^1$ and construct hypersurfaces of these bundles which form a degree $2$ del Pezzo fibration over $\P^1$ as a Mori fibre space. We classify all such hypersurfaces whose type $\III$ or $\IV$ Sarkisov links pass to a different Mori fibre space. A similar result for cubic surface fibrations over $\P^2$ is also presented.
\end{abstract}

\setcounter{tocdepth}{1}
\tableofcontents

\section{Introduction}
One possible outcome of the minimal model program is a Mori fibre space.
\begin{defi}\emph{A {\it Mori fibre space} is a contraction $\varphi\colon X\rightarrow S$, where
\begin{enumerate}[(1)]
\item $X$ is $\Q$-factorial with at worst terminal singularities,
\item $-K_X$ is $\varphi$-ample,
\item $\rho(X)=\rho(S)+1$,
\item $\dim S<\dim X$.\end{enumerate}}\end{defi}

Of course, by definition above, there are three cases of  $3$-dimensional Mori fibre spaces:
\begin{enumerate}[(i)]
\item $X$ is a Fano $3$-fold, when $\dim S=0$,
\item $X$ is a del Pezzo fibration, when $\dim S=1$,
\item $X$ is a conic bundle, when $\dim S=2$.\end{enumerate}

\begin{defi}\emph{Let $\varphi\colon X\rightarrow S$ and $\varphi^\prime\colon X^\prime\rightarrow S^\prime$ be Mori fibre spaces such that there is a birational map $f\colon X\dashrightarrow X^\prime $. The map $f$ is said to be {\it square} if there is a birational map $g\colon S\dashrightarrow S^\prime$, which makes the diagram 
\begin{center}$\xymatrixcolsep{2.8pc}\xymatrixrowsep{2.8pc}
\xymatrix{
X\ar[d]_{\varphi}\ar@{-->}[r]^f&X^\prime\ar[d]^{\varphi^\prime}\\
S\ar@{-->}[r]^g&S^\prime}$\end{center}
commute and, in addition, the induced birational map $f_L\colon X_L\rightarrow X^\prime_L$ between the generic fibres is biregular. In this situation, we say that the two Mori fibre spaces $X\rightarrow S$ and $X^\prime\rightarrow S^\prime$ are {\it birational square}.}\end{defi}

\begin{defi}\emph{A Mori fibre space $X\rightarrow S$ is {\it birationally rigid} if for any birational map $f\colon X\dashrightarrow X^\prime$ to another Mori fibre space $X^\prime\rightarrow S^\prime$, there exists a birational selfmap $\alpha\colon X\dashrightarrow X$ such that the composite $f\circ\alpha\colon X\dashrightarrow X^\prime$ is square.}\end{defi} 

In~\cite{CPR} it was shown that a general member in the list of $95$ families of Fano $3$-folds is birationally rigid. Birational rigidity of conic bundles has been studied by a number of people, for example see \cite{Iskovskikh2}, \cite{Iskovskikh1}, \cite{Sarkisov1}, \cite{Sarkisov2} and \cite{BCZ}. Del Pezzo fibrations split into $9$ cases according to the degree of the fibres, that is the intersection number $K_L^2$, where $L$ is the generic fibre. If the degree is greater than $5$, it is known that the $3$-fold is rational. Alexeev in~\cite{alexeev} proved that a standard degree $4$ del Pezzo fibration is birational to a conic bundle, and hence they are non-rigid. Rigidity of degree $3$ del Pezzo fibrations have been studies by many authors; for example see~\cite{pukhlikov} and \cite{BCZ}. Birational geometry of lower degree del Pezzo fibrations has been only studied in the smooth case. The main contributions being works of Pukhlikov~\cite{pukhlikov} and Grinenko~\cite{Grinenko-1, Grinenko-2, Grinenko-3}. In fact the smoothness condition of these varieties is very restrictive as in many cases the $3$-fold $X$ has nonsmooth terminal singularities. In that regard  most of the families constructed in this article have index $2$ singularities.

 We provide a natural construction for degree $2$ del Pezzo fibrations, denoted by $dP_2$. This is followed by classifying those which admit another Mori fibre space as a birational model (not birational square) where the other model is obtained by restriction of the $2$-ray game of the ambient space on the $3$-fold. In particular, these $3$-folds are non-rigid.\\




\paragraph{\bf Acknowledgement.} I am grateful to my supervisor Gavin Brown for introducing me to this problem, his constant support and many useful comments. This work has been supported by the EPSRC grant EP/E000258/1.


\section{Construction}

\begin{defi}\label{weighted bundle}\emph{A {\it weighted bundle over $\P^n$} is a rank 2 toric variety $\F=TV(A,I)$ defined by
\begin{enumerate}[(i)]
\item $\Cox(\F)=\C[x_0,\dots,x_n,y_0,\dots,y_m]$.
\item The irrelevant ideal of $\F$ is $I=(x_0,\dots,x_n)\cap(y_0,\dots,y_m)$.
\item and the $(\C^*)^2$ action on $\C^{n+m+2}$ is given by
\[A=\left(\begin{array}{ccccccc}
1&\dots&1&-\omega_0&-\omega_1&\dots&-\omega_m\\
0&\dots&0&1&a_1&\dots&a_m\end{array}\right)\qquad,\]
where $\omega_i$ are non-negative integers and $\P(1,a_1,\dots,a_m)$ is a weighted projective space.
\end{enumerate} }\end{defi}

\begin{defi}\label{weight-order}\begin{enumerate}[\emph{(a)}]
\item \emph{Let $T$ be a rank 2 toric variety. Suppose $t$ is a generating variable in the Cox ring of $T$ and that the action of the ${(\C^*)}^2$ on $t$ is given by $t\mapsto \lambda^a\mu^bt$, where $(a,b)\in\Z^2\backslash \{(0,0)\}$. We say that the number $\frac{a}{b}$ is the {\it ratio weight} of the variable $t$. Note that the ratio weight could be a rational number or $\infty=\frac{|a|}{0}$ or $-\infty=\frac{-|a|}{0}$.}
\item\emph{Let $T$ be a rank 2 toric variety with $\Cox(T)=\C[t_1,\dots,t_k]$. Define a total order $\preceq$ on $\{t_0,\dots,t_k\}$ by $t_i\preceq t_j$ if and only if the ratio weight of $t_j$ is less than or equal to the ratio weight of $t_i$. Note that we allow $-\infty$ and $\infty$ in their own right. If the ratio weight of $t_i$ is strictly bigger than the one for $t_j$, we write $t_i\prec t_j$.}
\end{enumerate}
\end{defi}
\begin{rmk}\emph{Note that the order $\preceq$ above is induced by the usual order in the set of extended real numbers in the {\bf reverse} direction!}\end{rmk}

Without loss of generality we can assume the variables of the $\Cox(\F)$ in Definition~\ref{weighted bundle} are in order with respect to $\preceq$. 
Let $Y_0,\dots,Y_r$ be the partition of ${y_0,\dots,y_m}$ such that variables contained in each $Y_i$ have the same ratio weight and that $Y_i$ is nonempty and contains all variables with that ratio weight. Furthermore we assume that they are in order with $Y_i\prec Y_{i+1}$, meaning the ratio weight of the variable in $Y_i$ is strictly bigger than the ratio weight of variables in $Y_{i+1}$. Note that this last condition makes $Y_0,\dots,Y_r$ a unique partition of ${y_0,\dots,y_m}$.

Consider the ideal $I_j=(x_0,\dots,x_n,Y_0,\dots,Y_{j-1})\cap(Y_j,\dots,Y_r) \subset \Cox(\F)$. Let $\F_j$ be the rank two toric variety defined by $TV(A,I_j)$, i.e.
\[\F_j=(\C^{n+m+2}\backslash V(I_j))\qt (\C^*)^2\]
in particular $\F_0=\F$. 
The following is an observation of the Theorem~4.1 in~\cite{BZ}, also known in~\cite{miles-toric}.

\begin{theorem}\label{thm!BZ} Let $\F/\P^n$ be a weighted bundle as before. Then the 2-ray link of $\F$ is given by one of the following:
\begin{enumerate}[(1)]
\item If $|Y_r|=1$, i.e. the set $Y_r$ has only one element,  then
\begin{center}$\xymatrixcolsep{3.5pc}\xymatrixrowsep{2.5pc}
\xymatrix{
&\F_0\ar[r]^{\Psi_1}\ar[ld]_{\Phi}&\F_1\ar[r]^{\Psi_2}&\dots\ar[r]^{\Psi_{r-1}}&\F_{r-1}\ar[rd]^{\Phi^\prime}\\
\P^1&&&&&\F_r}$\end{center}
 where $\F_0=\F$, $\Psi_i$ are isomorphisms in codimension one and $\Phi^\prime$ is a divisorial contraction.
 
 \item If $|Y_r|>1$, then
\begin{center}$\xymatrixcolsep{3.5pc}\xymatrixrowsep{2.5pc}
\xymatrix{
&\F_0\ar[r]^{\Psi_1}\ar[ld]_{\Phi}&\F_1\ar[r]^{\Psi_2}&\dots\ar[r]^{\Psi_{r}}&\F_{r}\ar[rd]^{\Phi^\prime}\\
\P^1&&&&&\P}$\end{center}
 where $\F_0=\F$, $\Psi_i$ are isomorphisms in codimension one, $\Phi^\prime$ is a fibration and $\P=\P(a_{r_1},\dots,a_{r_k})$, where $a_{r_1},\dots,a_{r_k}$ are the denominators of the ratio weights of the variables in $Y_r$.
\end{enumerate}
\end{theorem}

Note that case $(1)$ in this theorem is the Type~$\III$ Sarkisov link of $\F$ and case $(2)$ is the Type~$\IV$.

\begin{defi}\label{Xi} \emph{ Let $\F/\P^n$ be a weighted bundle as in Definition~\ref{weighted bundle}, and $\F_i$ be the varieties appearing in its 2-ray link of Theorem~\ref{thm!BZ}.
Let $\overline{X}\colon(f=0)\subset\C^{n+m+2}$ be a hypersurface in $\C^{n+m+2}$, the Cox cover of $\F$, defined by $f\in\C[x_0,\dots,x_n,y_0,\dots,y_m]$. Assume $f$ is irreducible, reduced and homogeneous with respect to the action of ${(\C^*)}^2$. Define $X_i\subset\F_i$ to be  
\[X_i=(\overline{X}\backslash V(I_i))/{(\C^*)}^2\]
and let $\psi_i$ (respectively $\varphi$, $\varphi^\prime$) be the restriction of $\Psi_i$ (respectively $\Phi$, $\Phi^\prime$) to $X_{i-1}$. Then we say $X_0$ has an {\it $\F$-link} if
\begin{enumerate}[(i)]
\item $\psi_i$ are isomorphisms in codimension one (possibly isomorphisms).
\item $\varphi$ and $\varphi^\prime$ are extremal contractions.\end{enumerate} }\end{defi}
In other words, $X_0$ has an $\F$-link if the 2-ray game of $X_0$ is obtained by the restriction of the 2-ray game of $\F_0$ (although some $\varphi_i$ may be isomorphisms and hence redundant from the game). If in addition, each $X_i$ is $\Q$-factorial with terminal singularities, then we say $X_0$ has an {\it $\F$-Sarkisov link}. 

\section{Sarkisov links from general $dP_2\slash\P^1$ hypersurfaces}

We consider weighted bundles over $\P^1$ with fibre $\P(1,1,1,2)$; these are a natural place to embed 3-fold degree 2 del Pezzo fibrations.

\begin{defi}\label{dP2-definition}\emph{ A $3$-fold $X$ is a {\it degree 2 del Pezzo fibration over $\mathbb{P}^1$} (denoted by $dP_2$ fibration, or simply $dP_2\slash\P^1$) if $X$ has an extremal contraction of fibre type $\varphi \colon X\rightarrow\mathbb{P}^1$ such that
\begin{enumerate}[(a)]
\item $X$ has at worst terminal singularities and is $\Q$-factorial.
\item The nonsingular fibres of $\varphi$ are del Pezzo surfaces of degree two.
\end{enumerate}}
\end{defi}

Let $\mathcal{F}$ be a rank two toric variety defined by $\mathcal{F}=TV(I,A)$, where $I\subset\mathbb{C}[u,v,x,y,z,t]$ is the irrelevant ideal $I=(u,v)\cap(x,y,z,t)$ and $A$ is the representing matrix of the action of $\mathbb{C}^*\times\mathbb{C}^*$ given by

\begin{equation}\label{Amat}
A=\left(\begin{array}{cccccc}
1&1&-\alpha&-\beta&-\gamma&-\delta\\
0&0&1&1&1&2\end{array}\right)\quad.\end{equation}

\begin{rmk}\emph{ Up to the action of $\Sl(2,\mathbb{Z})$, any matrix of type \eqref{Amat} can be written uniquely in one of the following forms:
\[\begin{array}{lllll}
\vspace{0.4cm}
(i)&\qquad& A=\left(\begin{array}{cccccc}
1&1&0&-a&-b&-c\\
0&0&1&1&1&2\end{array}\right)&\qquad& 0<c\,,\,0\le a\le b\\
\vspace{0.4cm}
(ii)&\qquad& A=\left(\begin{array}{cccccc}
1&1&-a&-b&-c&0\\
0&0&1&1&1&2\end{array}\right)&\qquad &\hspace{0.42cm} 0\le a\le b\le c\\
(iii)&\qquad& A=\left(\begin{array}{cccccc}
1&1&-a&-b&-c&-1\\
0&0&1&1&1&2\end{array}\right)&\qquad& \hspace{0.42cm} 0<a\leq b\leq c\quad.\end{array}\]}\end{rmk}

The Picard group of $\mathcal{F}$ is isomorphic to $\mathbb{Z}^2$. Let $L$ and $M$ be Weil divisors of $\F$ with weights $(1,0)$ and $(0,1)$. For example in the case $(i)$ above $u\in H^0(\F,L)$ and  $x\in H^0(\F,M)$. A simple toric singularity analysis shows that $\F$ is smooth away from the curve $\Gamma_t=(x=y=z=0)$. The curve $\Gamma_t$ is a rational curve with singularity of transverse type $\frac{1}{2}(1,1,1)$ along $\Gamma_t$.

Let $D= 4M-eL\in\Div(\F)$ be a divisor in $\mathcal{F}$ and $X=(f=0)\subset\F$ be the hypersurface of $\F$ defined by a general $f\in H^0(\F,D)$.
We say that $X\subset\F$ has bi-degree $(-e,4)$ and encode these information about $X$ and $\F$ with the notation
\[\left(\begin{array}{c}-e\\4\end{array}\right)\subset \left(\begin{array}{cccccc}
1&1&-\alpha&-\beta&-\gamma&-\delta\\0&0&1&1&1&2\end{array}\right)\quad.\]

The goal is to find conditions on $X$ and $\mathcal{F}$ such that $X$ is a Mori fibre space, whose generic fibre is a del Pezzo surface of degree 2, that has an $\F$-Sarkisov link to another Mori fibre space.

\subsection{The main result}

\begin{theorem}\label{list}\label{th!dP2main}
 Consider a hypersurface $X\subset \F$ with 
\[\left(\begin{array}{c}-e\\4\end{array}\right)\subset\left(\begin{array}{cccccc}
1&1&-\alpha&-\beta&-\gamma&-\delta\\0&0&1&1&1&2\end{array}\right)\quad,\]
where the weights $\alpha,\beta,\gamma$ are normalised with $\gamma\ge\beta\ge\alpha\geq 0$ and $\delta\ge 0$. 
Suppose the Type~\III\ or \IV\ 2-ray game of $\F$ restricts to a Sarkisov link for $X$.
Then the weights $\alpha,\beta,\gamma,\delta,e$ are among those appearing in the left-hand column of Table~\ref{table!dP2list}.

Moreover, we show in \ref{geometry-main-theorem} below that if $X$ is a general hypersurface of type $(\alpha,\beta,\gamma,\delta;e)$ from table~\ref{table!dP2list}, then $X$ is nonrigid. The Sarkisov link to another Mori fibre space is described in the remaining columns of Table~\ref{table!dP2list}.

\begin{table}[ht]\small
\[\begin{array}{cc||c|c|c|c}
\text{No.}&(\alpha,\beta,\gamma,\delta;e)&\psi_1&\psi_2&\varphi^\prime&\text{new model}\\
\hline
\hline
1&(0,0,0,0;-1)&\text{n/a}&\text{n/a}&\text{contraction}&\mathbb{P}(1,1,1,2)\\
\hline
2&(0,0,0,1;0)&\text{n/a}&\text{n/a}&\text{contraction}&Y_4\subset\mathbb{P}(1,1,1,2,2)\\
\hline
3&(0,0,1,0;0)&\text{n/a}&\text{n/a}&\text{contraction to a line}&Y_4\subset\mathbb{P}(1,1,1,1,2)\\
\hline
4&(0,1,1,0;0)&\text{flop of $2\times\P^1$}&\text{n/a}&\text{fibration}&dP_2 \text{ fibration}\\
\hline
5&(0,0,1,1;0)&\text{flop of 4} &\text{n/a}&\text{divisorial contraction}&Y_4\subset\mathbb{P}^4 \\
 & &\text{disjoint } \P^1& &\text{to a point}&\\
\hline
6&(1,1,1,1;2)&\cong&\text{n/a}&\text{fibration}&\text{conic bundle with}\\
&&&&&\text{discriminant }\Delta_8\subset\P^2\\
\hline
7&(0,1,1,1;1)&\text{flop}&\text{flip}&\text{fibration}&dP_3 \text{ fibration}\\
\hline
8&(0,1,1,2;2)&\text{flop}&\text{n/a}&\text{fibration}&\text{conic bundle over}\\ 
 && & &&\mathbb{P}(1,1,2)\text{ with}\\
 &&&&&\text{disc. }\Delta_{10} \subset \P(1,1,2)\\
\hline
9&(0,1,2,1;2)&\text{flop}&\cong&\text{contraction}&Y_6\subset\mathbb{P}(1,1,1,2,3)\\
\hline
10&(0,1,1,3;3)&\text{flop}&\text{n/a}&\text{contraction}&Y_6\subset\mathbb{P}(1,1,2,2,3)\\
\hline
11&(0,2,2,1;2)&\text{anti-flip}&\cong&\text{fibration}&dP_1 \text{ fibration}\\
\hline
12&(0,1,2,3;3)&\text{anti-flip}&\text{flop}&\text{contraction}&Y_5\subset\mathbb{P}(1,1,1,1,2)\\
\hline
13&(0,1,2,4;4)&\text{anti-flip}&\cong&\text{fibration}&dP_2 \text{ fibration over }\mathbb{P}(1,2)\\
\end{array}\]
\caption{{\footnotesize Data of Type~\III\ and \IV\ links from general degree~2 del Pezzo hypersurface fibrations\label{table!dP2list}}}
\end{table}
\end{theorem}


\section{General hypersurfaces}
In this section, we prove the constructive part, the second part, of the Theorem~\ref{th!dP2main} in one direction by  calculating the birational link for a general hypersurface in each family in Theorem~\ref{th!dP2main} and then we show in subsection~\ref{Pic-section} that these hypersurfaces are indeed $dP_2\slash\P^1$. These links are provided from the restriction of the natural 2-ray game of the ambient toric variety $\F$ to $X$.

\subsection{Geometry of the links}
\label{geometry-main-theorem}
 In order to match the notation of Theorem~\ref{thm!BZ}, in each case we rewrite the defining numerical system, normalised by the order $\preceq$, and give the numerical system of the rank 2 variety at the end of each link. Rather than following the order in Table~\ref{table!dP2list}, we analyse cases together according to the structures at the end of their links.

\subsubsection{Links to conic bundles}
\begin{description}
\item[Family~6] $u=v\prec t\prec x=y=z$\\
The 2-ray game of $\F$ starts by $\Psi_1$, which is a flip of type $(2,2,-1,-1,-1)$ in the neighbourhood $(t\neq 1)$ of the flipping curve~ $\P^1_{u:v}$. The second and final step of the 2-ray game is a $\P^2$ fibration to $\P^2_{x:y:z}$. Considering $X$ of bi-degree $(-2,4)$, the Newton polygon of $X$ is
\[\begin{array}{c|c}
\deg \text{ of }u,v\text{ coefficient}&\\
\hline
0&t^2\\
1&tx^2\quad txy\quad \dots\quad tyz\quad tz^2\\
2&x^4\quad x^3y\quad\dots\quad yz^3\quad z^4\quad.\\
\end{array}\]
This means that $f$, the defining polynomial of $X$, includes terms of the form $t^2$ and $l(u,v)tx^2$ and $q(u,v)x^4$, where $l(u,v)$ is a general linear form in $u,v$ and $q(u,v)$ is a general quadratic.  We use the notation $t^2\in F$ to say that the monomial $t^2$ appears as a term of $f$.
It is also useful for us to describe $f$ as the product of the following matrices:
\begin{equation}\label{conic1}\left(\begin{array}{ccc} u&v&t \end{array}\right)\left(\begin{array}{ccc}
*_4&*_4&*_2\\
*_4&*_4&*_2\\
*_2&*_2&1\end{array}\right)\left(\begin{array}{c} u\\v\\t\end{array}\right)\qquad,\end{equation}
where by $*_k$ we mean a general homogeneous polynomial of degree $k$ in variables $x,y,z$.

Having the monomial $t^2\in f$ ensures that $X$ does not intersect with the singular locus of $\F$ as $\Sing(\F)=\Gamma_t$. Having this key monomial also shows that $\psi_1$, the restriction of $\Psi_1$ to $X$, is an isomorphism on $X$. The restriction of $\Phi^\prime$ to $X$ defines a fibration to $\P^2_{x:y:z}$ with fibres being conic curves. The discriminant of this conic is the determinant of the $3\times 3$ matrix in \eqref{conic1}. The degree of the discriminant in this case is 8.\\

\item[Family~8] $u=v\prec x\prec y=z=t$\\
Let us describe the birational geometry of the ambient space $\F$. The 2-ray game of $\F$ starts by mapping to $\P^1$ in one side (the given extremal contraction) and anti-flip $(1,1,-1,-1,-2)$ in the other side. This anti-flip can be read by fixing the action of the second component of the $(\C^*)^2$ in the neighbourhood $(x\neq 0)$ by putting $x=1$. Then the game follows by an extremal contraction of fibre type to $\P(1,1,2)$. To restrict this toric 2-ray game to $X$, we need to know $f$, the defining polynomial of $X$, which can be seen from the Newton polygon of $X$, 
\[\begin{array}{c|c}
\deg \text{ of }u,v\text{ coefficient}&\\
\hline
0&x^2t\quad xy^2\quad xyz\quad xz^2\\
1&xyt\quad xzt\quad xy^3\quad xy^2z\quad xyz^2\quad xz^3\\
2&y^2t\quad yzt\quad z^2t\quad t^2\qquad.\\
\end{array}\]
 Here our essential terms in $f$ are $x^2t$ and $q(u,v)t^2$, where $q(u,v)$ is a general quadratic in $u,v$. Having $q(u,v)t^2\in f$ means that the singular locus of (a general quasismooth) $X$ is the intersection of $X$ with $\Gamma_t$, which in this case is only two points $(q=0)\cap\Gamma_t$. 
 
 The $\F$-Sarkisov link of a general $X$ in this family, starts by an Atiyah flop and follows by a fibration to $\P(1,1,2)$ with conic curve fibres. The flop is the restriction of the $(1,1,-1,-1,-2)$ anti-flip on $\F$. The restriction is a flop because the monomial $x^2t\in f$ allows us to eliminate the variable $t$ in the neighbourhood $(x\neq 0)$. 
 
 Similar to the previous case, considering the defining polynomial of $X$ in the form  
\begin{equation}\label{conic2}\left(\begin{array}{ccc} u&v&t \end{array}\right)\left(\begin{array}{ccc}
*_4&*_4&*_3\\
*_4&*_4&*_3\\
*_3&*_3&*_2\end{array}\right)\left(\begin{array}{c} u\\v\\t\end{array}\right)\end{equation}
tells us that the degree of the discriminant of the conic in this case is 10.

\begin{rmk} \emph{In \cite{mori-prokh}, a list of possible singularities that the base variety of a conic bundle can admit is provided. By Theorem~1.2.7. in~\cite{mori-prokh}, $\P(1,1,2)$ is a legal base since it has only a quotient singularity $\frac{1}{2}(1,1)$, which is Du Val.}\end{rmk}  
\end{description}

\subsubsection{Links to del Pezzo fibrations}
\begin{description}
\item[Family~4] $u=v\prec x=t\prec y=z$\\
The 2-ray game of $\F$ in this case is represented by
\begin{center}$\xymatrixcolsep{3pc}\xymatrixrowsep{2pc}
\xymatrix{
&\F\ar[rd]^{\Psi_1^-}\ar[ld]_{\Phi}&&\F_1\ar[rd]^{\Phi^\prime}\ar[ld]_{\Psi_1^+}\\
\P^1_{u:v}&&\mathcal{G}&&\P^1_{y:z}}\qquad,$\end{center}
where the composition map $\Psi_1=(\Psi_1^+)^{-1}\circ\Psi_1^-$, is a toric 4-fold flop. Both $\Psi_1^-$ and $\Psi_1^+$ are isomorphism away from $\P^1\times\P^1$. The first map, $\Psi_1^-$, contracts the surface $\P^1_{u:v}\times\P_{x:t}(1,2)$ to $\P^1_{x:t}$ and $\Psi_1^+$ contracts $\P^1_{y:z}\times\P^1_{x:t}$ to the same line. This composition defines $\Psi_1$ as a toric 4-fold flop. The next step of the 2-ray game, $\Phi^\prime$ provides a fibration to $\P^1_{y:z}$ with fibres isomorphic to $\P(1,1,1,2)$.

The defining equation of $X$ has the form $f=g+h$, where $g=g(x,t)$ is a quartic in variables $x$ and $t$ only. This ensures that the restriction of $\Psi_1^-$ contracts two disjoint $\P^1$, defined by $(g=0)\cap\P^1_{u:v}\times\P_{x:t}(1,2)$ to two points in $\P^1_{x:t}$, namely the solutions of $(g=0)\subset\P(1,2)$. This argument shows that $\psi_1$ is formed of a flop $\psi_1\colon X\rightarrow X_1$, which flops two disjoint copies of $\P^1$. At the end of the link, the restriction of $\Phi^\prime$ to $X_1$ provides the extremal contraction of fibre type to $\P^1$ with degree 2 del Pezzo fibres.

\item[Family~7] $u=v\prec x\prec t\prec y=z$\\
This case is similar to the previous one and the result was already found in \cite{BCZ}. A full analysis is given in~\cite{BCZ}~Family~5,~\S4.4.2.\,.

\item[Family~11] $u=v\prec x\prec t\prec y=z$\\
The diagram of the 2-ray game of $\F$ is
\begin{center}$\xymatrixcolsep{3pc}\xymatrixrowsep{2pc}
\xymatrix{
&\F\ar[r]^{\Psi_1}\ar[ld]_{\Phi}&\F_1\ar[r]^{\Psi_2}&\F_2\ar[rd]^{\Phi^\prime}&\\
\P^1_{u:v}&&&&\P^1_{y:z}}\qquad,$\end{center}
where $\Psi_1$ is the anti-flip $(1,1,-1,-2,-2)$ flipping a copy of $\P^1$ to $\P(1,2,2)$. In particular, the flipping locus of $\F_1$ has line of singularity of transverse type $\frac{1}{2}(1,1,1)$. Note that $\F$ contains a singular line $\Gamma_t$, which is preserved by $\Psi_1$. The second anti-flip $\Psi_2$, is of type $(2,2,1,-3,-3)$, which flips a surface $\P(1,2,2)$ (including $\Gamma_t$) to a singular curve of transverse type $\frac{1}{3}(1,2,2)$. $\Phi^\prime\colon\F_2\rightarrow\P^1$ is a fibration, with $\P(1,1,2,3)$ fibres.

Now we consider the restriction of this game to $X$. The essential monomials of the defining polynomial of $X$ are $t^2$ and $x^3y$. The first monomial, $t^2$ shows that $\Gamma_t\cap X$ is empty for a general $X$. In fact, Bertini Theorem implies that $X$ is smooth as the base locus of the linear system $D$ includes only the curve $\Gamma_x=(u_0:v_0;1:0:0:0)$, which is guaranteed to be smooth by $x^3y\in f$.

 The restriction of $\Psi_1$ to $X$ is a Francia anti-flip as we can eliminate the variable $y$ in a neighbourhood of the flipping curve using $x^3y$ and implicit function theorem. Note that the variety $X_1$ has a $\frac{1}{2}(1,1,1)$ singularity obtained by this anti-flip. The restriction of $\Psi_2$ to $X_1$ is an isomorphism as $t^2\in f$. And finally, $\varphi^\prime\colon X_1\rightarrow \P^1$ is a Mori fibre space with generic fibre isomorphic to a del Pezzo surface of degree 1. 

\item[Family~13] $u=v\prec x\prec y\prec z=t$\\
A similar argument shows that the general $X$ in this case, after a Francia anti-flip has an extremal contraction of fibre type to $\P(1,2)$, with generic fibre isomorphic to a degree~2 del Pezzo surface.
\end{description}

\subsubsection{Links to Fano 3-folds}
\begin{description}
\item[Family~1] $u=v\prec x=y=z=t$\\
The defining polynomial of a general $X$ in this case is of the form $uf_4(x,y,z,t)=vg_4(x,y,z,t)$, for general degree 4 polynomials $f$ and $g$ in variables $x,y,z,t$. The 2-ray game of $\F$ is continued by a fibration $\Phi^\prime$ to $\P(1,1,1,2)$ with $\P^1$ fibres. The restriction of this map to $X$ provides $\varphi^\prime:X\rightarrow\P(1,1,1,2)$, which contracts the divisor $(f=g=0)\subset X$ to a curve in $\P(1,1,1,2)$, defined by the same set of equations.

\item[Family~2]$u=v\prec x=y=z\prec t$\\
The 2-ray game of the ambient toric variety is described by
\begin{center}$\xymatrixcolsep{3pc}\xymatrixrowsep{2pc}
\xymatrix{
&\F\ar[rd]^{\Phi^\prime}\ar[ld]_{\Phi}&\\
\P^1_{u:v}&&\P&\hspace{-1.4cm}(1,1,1,2,2)}\qquad,$\end{center}
where $\Phi^\prime$ is the divisorial contraction defined by the basis of the Riemann-Roch space of the divisor $D_x\sim(x=0)$. More precisely, the equation of $\Phi^\prime$ is
\[\qquad\Phi^\prime\colon \F\rightarrow\P(1,1,1,2,2)\]
\[(u:v;x:y:z:t)\mapsto(x:y:z:ut:vt)\qquad.\]
It is clear from this equation that the divisor $(t=0)$ is contracted to the surface $\P^2_{x:y:z}$. Note that this map has no base point, as the locus where all these monomials vanish is precisely the Cox irrelevant ideal of $\F$, i.e. $(u,v)\cap(x,y,z,t)$. 

The equation of a general $X$ in this family is of the form $t^2q(u,v)=f(x,y,z)+\dots$, where $q$ is a quadratic polynomial in $u,v$ and $f$ is a quartic with variables $x,y,z$. Such $X$ has two singular points of type $\frac{1}{2}(1,1,1)$, which are located at the intersection of $X$ with $\Gamma_t$, that is the solutions of $(q=0)\cap\Gamma_t$. Then $X$ follows the 2-ray game of the ambient space by contracting the divisor $(t=0)$ to the curve $(f=0)\subset\P^2_{x:y:z}$ on an index 3 Fano 3-fold defined by $X_4\subset\mathbb{P}(1,1,1,2,2)$.\\

The equation of the Fano 3-fold, the image of $X$ under this map, can be derived explicitly using this coordinate map. For example if the coordinate variables on $\P(1,1,1,2,2)$ are $x,y,z,u^\prime,v^\prime$, then this Fano variety is the hypersurface defined by
\[q(u^\prime,v^\prime)=f(x,y,z)+\dots\qquad.\]

\begin{cor} An index $4$ Fano $3$-fold hypersurface $Y_4\subset\P(1,1,1,2,2)$ is birational to a degree 2 del Pezzo fibration over $\P^1$.\end{cor}

\item[Family~3]$u=v\prec x=y=t\prec z$\\
Analysis of the link is similar to the previous case with the final divisorial contraction $\Phi^\prime$ with equation
\[(u;v;x:y:t:z)\mapsto(uz:vz:x:y:t)\qquad.\]
The image of $X$ under this map is an index $2$ Fano hypersurface defined by a quartic in $\P(1,1,1,1,2)$. 
\begin{cor} An index $2$ Fano $3$-fold hypersurface $Y_4\subset\P(1,1,1,1,2)$ is birational to a degree 2 del Pezzo fibration over $\P^1$.\end{cor}

\item[Family~5] $u=v\prec x=y\prec t\prec z$\\
The 2-ray game of $\F$ starts by a flop and continues by a divisorial contraction to $\P^4$. The toric flop contracts a copy of $\P^1\times\P^1$ to $\P^1$ and extracts another $\P^1\times\P^1$. The restriction of this birational map to $X$ flops 4 analytically disjoint copies of $\P^1$, since the defining polynomial of $X$ includes a quartic in the $x,y$ variables.
 
A general $X$ in this family is singular at two points of type $\frac{1}{2}(1,1,1)$. As usual, these points are the locus where $X$ meets $\Gamma_t$. In fact we can assume that the defining polynomial of $X$ is of the form $(u^2+v^2)t^2+f(x,y)+\dots$, where $f$ is a general quartic in $x,y$.
The divisorial contraction has the coordinate description
\[(u:v;x:y:t:z)\mapsto(uz^2:vz^2:xz:yz:t)\qquad,\]
which shows that the divisor $(z=0)$ gets contracted to the point $p_t\in\P^4$. The equation near this point has a local type $u^2+v^2+x^4+y^4$. In other words this point is terminal. In fact this example was already known to be nonrigid. See \cite{CPR}, Example~7.5.1. 

\item[Family~9] $u=v\prec x\prec t\prec y\prec z$\\
The 2-ray game on the ambient space is
\begin{center}$\xymatrixcolsep{2.6pc}\xymatrixrowsep{2pc}
\xymatrix{
&\F\ar[r]^{\Psi_1}\ar[ld]_{\Phi}&\F_1\ar[r]^{\Psi_2}&\F_2\ar[rd]^{\Phi^\prime}&\\
\P^1_{u:v}&&&&\P&\hspace{-1.25cm}(1,1,1,2,3)&,}$\end{center}
where $\Psi_1$ is the anti-flip $(1,1,-1,-1,-2)$ and $\Psi_2$ is the flip $(2,2,1,-1,-3)$. The final contraction is
\[\Phi^\prime\colon(u:v;x:t:y:z)\mapsto(u_0:v_0:y:x_0:z_0)=(uz:vz:y:xz:tz)\qquad,\]
which is the ordinary blow up of the smooth point $p_y\in\P(1,1,1,2,3)$.
The Newton polygon of $X$ in this family is described by
\vspace{-0.15cm}
\[\begin{array}{c|c}
\deg \text{ of }u,v\text{ coefficient}&\\
\hline
0&t^2\quad x^3z\quad x^2y^2\quad xty\\
1&xy^3\quad x^2yz\quad xtz\quad ty^2\\
2&y^4\quad xy^2z\quad tyz\quad x^2z^2\\
3&xyz^2\quad tz^2\quad y^3z\\
4&xz^3\quad y^2z^2\\
5&yz^3\\
6&z^4\qquad.
\end{array}\]
Having the term $t^2\in f$, the defining polynomial of $X$, guarantees smoothness of $X$. The map $\psi_1$, the restriction of $\Psi_1$ to $X$, is an Atiyah flop as the variable $z$ can be eliminated in a neighbourhood of the flopping curve using the monomial $x^3z$ and the implicit function theorem. Similarly, we can observe that $\psi_2$ is an isomorphism as $t^2\in f$. The image of $X_1$ under $\varphi^\prime$ is an index 2 Fano hypersurface $Y$ defined by a degree 6 polynomial in $\P(1,1,1,2,3)$. One can see that under this map, the divisor $(z=0)$ goes to the point $p_y\in Y$. This point is a $cA_1$ point as the defining polynomial of $Y$ is 
\[t_0^2+x_0^3+y^4u_0v_0+u_0^6+v_0^6+\dots\qquad.\]

Conversely, a general Fano hypersurface $Y_6\subset\P(1,1,1,2,3)$ with a $cA_1$ point is birational to a degree 2 del Pezzo fibration over $\P^1$.

\item[Family~10] $u=v\prec x\prec y=z\prec t$\\
The 2-ray game on $\F$ is
\begin{center}$\xymatrixcolsep{3pc}\xymatrixrowsep{2.5pc}
\xymatrix{
&\F_0\ar[r]^{\Psi_1}\ar[ld]_{\Phi}&\F_1\ar[rd]^{\Phi^\prime}\\
\P^1&&&\P&\hspace{-1.4cm}(1,1,2,2,3)&,}$\end{center}
where $\Psi_1$ is the anti-flip $(1,1,-1,-1,-3)$. And the final contraction is $\Phi^\prime\colon\F_1\rightarrow\P(1,1,2,2,3)$ defined by
\[(u;v;x:y:z:t)\mapsto(y:z:u_0:v_0:x_0)=(y:z:ut:vt:xt)\qquad.\]
This map contracts the divisor $(t=0)$ on $\F_1$ to the line $\P^1_{y:z}\subset\P(1,1,2,2,3)$. 

The Newton polygon of a general $X$ in this family is
\[\begin{array}{c|c}
\deg S^k(u,v,w)&\\
\hline
0&x^2t\quad xy^3\quad xy^2z\quad xyz^2\quad xz^3\\
1&y^4\quad\dots z^4\quad xyt\quad xzt\\
2&y^2t\quad yzt\quad z^2t\\
3&t^2\qquad.
\end{array}\]
The coefficient of $t^2$ in the equation indicates that
\[\Sing(X)=\Gamma_t\cap X= 3\times\frac{1}{2}(1,1,1)\quad.\]
The map $\psi_1$, obtained by restricting $\Psi_1$ to $X$ is a flop $(1,1,-1,-1)$, as we are able to eliminate the variable $t$ near the flopping curve using the monomial $x^2t$. 
The map $\varphi^\prime$ contracts the divisor $(t=0)\subset X_1$ to the line $\P^1_{y:z}$ on an index 3 Fano variety $Y$ defined by a degree 6 polynomial in $\P(1,1,2,2,3)$. The defining polynomial of $Y$ is
\[x_0^2+g_3(u_0,v_0)+uq_4(y,z)+vq^\prime_4(y,z)+\dots\qquad,\]
where $g_3$ is a general cubic in the variables $u_0,v_0$; $q$ and $q^\prime$ are general quartics in $y,z$. Hence $Y$ is smooth along $\P^1_{y:z}$ and has only 3 singular points of type $\frac{1}{2}(1,1,1)$, namely at the solutions of $(g_3=0)$.

\item[Family~12] $u=v\prec x\prec y\prec t\prec z$\\
The 2-ray game of $\F$ is represented in the diagram:
\begin{center}$\xymatrixcolsep{2.6pc}\xymatrixrowsep{2pc}
\xymatrix{
&\F\ar[r]^{\Psi_1}\ar[ld]_{\Phi}&\F_1\ar[r]^{\Psi_2}&\F_2\ar[rd]^{\Phi^\prime}&\\
\P^1_{u:v}&&&&\P&\hspace{-1.25cm}(1,1,1,1,2)&,}$\end{center}
where $\Psi_1$ is the anti-flip $(1,1,-1,-3,-2)$ and $\Psi_2$ is the smooth flip $(1,1,1,-1,-1)$. The singular locus of $X$ is characterised by the coefficient of $t^2\in f$; this is a cubic in $u,v$, so for $X$ general $\Sing(X)=3\times\frac{1}{2}(1,1,1)$. The map $\psi_1$, the restriction of $\Psi_1$ to $X$, is the Francia anti-flip as the variable $t$ can be eliminated in a neighbourhood of $\Gamma_x=(u_0:v_0;1:0:0:0)$ using the monomial $x^2t$. Similarly, using the monomial $xy^3$, we can eliminate the variable $x$ in a neighbourhood of the flipping locus of $\Psi_2$ and observe that $\psi_2$ is an Atiyah flop. The final map $\varphi^\prime$, contracts the divisor $(z=0)$ to a point on an index 1 Fano hypersurface defined by a degree 5 polynomial in $\P(1,1,1,1,2)$. Note that this Fano hypersurface is quasi-smooth away from the image of contraction, which is a $cD_4$ singularity as it is locally defined by $x^2+u^3+v^3+y^4$. 
It was shown in \cite{CPR} that a general quasi-smooth Fano hypersurface of this type is birationally rigid.

\end{description}

\subsection{Mobile cones}\label{mobile-cones}
The aim is to prove that all varieties listed in Table~\ref{table!dP2list} satisfy the conditions of Definition~\ref{dP2-definition}. In fact the only remaining part to check is the Picard number. This is done in~\ref{Pic-section}. On the other hand, we must prove that this is the complete list; meaning any $dP_2\slash\P^1$ which does not appear in this list cannot have a link to another Mori fibre space following the 2-ray game of $\F$. Therefore we compute various cones of $X$ and $\F$ that we need later.

\begin{prop}\label{generators-of-mob} Let $\F$ be the toric variety described in~\ref{general-F}. Then 
\begin{enumerate}[(i)]
\item\label{eff-cone-F} the pseudo-effective cone of $\F$ is generated by $D_u$ and $D_4$, and 
\item\label{mob-cone-F} the mobile cone $\MMob(\F)$ is generated by $D_u$ and $D_3$,
\end{enumerate}
where $D_u,D_v$ and $D_i$ are divisors defined by $(u=0),\,(v=0)$ and $(x_i=0)$.\end{prop}
\begin{proof}
The fact that the Picard number of $\F$ is $\rho(\F)=2$ allows one to write $\N^1(\F)_\R\cong\R^2$ and hence draw all these cones in the plane
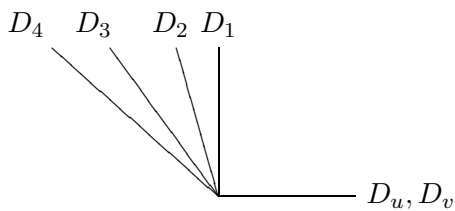
\[\xygraph{
!{(0,0) }="a"
!{(2,0) }*+{D_u,D_v}="b"
!{(0,1.8) }*+{D_1}="c"
!{(-0.5,1.8) }*+{D_2}="d"
!{(-1.3,1.8) }*+{D_3}="e"
!{(-2,1.8) }*+{D_4}="f"
"a"-"b" "a"-"c" "a"-"d" "a"-"e" "a"-"f"
}  \]
The rays are labelled by divisors that lie on them away from the origin.
Note that the rays correspond to some $D_i$ and $D_j$ might coincide. This is exactly when $x_i=x_j$ . 

Obviously $\left<D_u,\dots,D_4\right>\subset\NNE^1(\F)$. We show that any prime divisor corresponding to a lattice point in the plane outside of  this cone is not numerically equivalent to an effective divisor. Any divisor given by a lattice point in $\R^2-\NNE^1(\F)$ is numerically equivalent to a divisor $A$, $A^\prime$ or $A^{\prime\prime}$, where
\[\begin{array}{lll}
A=-\mu D_u+\lambda D_4&&\text{for }\mu>0,\,\lambda\geq 0,\\
A^\prime=-\mu D_u-\lambda D_4&&\text{for }\mu>0,\,\lambda> 0,\\
A^{\prime\prime}=\mu D_u-\lambda D_4&&\text{for }\mu\geq 0,\,\lambda> 0.
\end{array}\]

We show that $A$ cannot be effective. Define a curve $l=(x_1=x_2=x_3=0)\subset\F$, where without loss of generality $b_4=1$. We have
\[A\cdot l=-\mu D_u\cdot l+\lambda D_4\cdot l=-\mu<0\qquad.\]
Since $A$ is prime, we must have $l\subset A$.
Now consider the family of curves defined by the ideal \[I_C=(x_1,x_2+\varphi_{\delta-\beta}(u,v)x_4,x_3\psi_{\delta-\gamma}(u,v)x_4)\quad.\] For any curve $C$ in this family and any divisor $D$ on $\F$, there exists a positive rational number $r$ such that $r(l\cdot D)=C\cdot D$. Hence The support of this family lies in $A$. On the other hand, it is easy to see that for any point in $D_1$ there is a curve $C$ in this family which contains that point. In other words, $D_1$ is contained in the support of this family and hence $D_1\subset A$. But $A$ is prime and this is a contradiction.

Proofs for the other two cases, $A^\prime$ and $A^{\prime\prime}$ are similar and we do not write them here.\\

In order to prove (\ref{mob-cone-F}), we must show that the cone generated by $D_u$ and $D_3$ is the $\MMob(\F)$. The divisor $D_u$ is mobile as $D_v\in|D_u|$ and hence this linear system is base point free. Any effective divisor $\Q$-linearly equivalent to $D_3$ is of the form $\lambda D_4+\mu D_i$ or $\lambda D_4+\mu D_u$ for some positive integers $\lambda$ and $\mu$. Therefore $\Bs(D_3)\subset (x_3=x_4=0)$, and hence $|D_3|$ has no fixed component; the fixed part has  codimension at least two. This shows that $\left<D_u,D_3\right>\subset\MMob(\F)$. To complete the proof we must show that any effective divisor in $\NNE^1(\F)-\MMob(\F)$ is not mobile. But any such divisor is numerically equivalent to a divisor of the form $\mu D_3+\lambda D_4$ for some non-negative integers $\mu$ and $\lambda$. The fixed part of the linear system of such divisor includes $D_4$ and hence this divisor cannot be mobile.
\end{proof}

\begin{defi}\label{def-MDS}\emph{(\cite{Hu}, Definition~1.10) A normal projective variety $X$ is called a {\it Mori dream space} if
\begin{enumerate}[(i)]
\item $X$ is $\Q$-factorial and $\Pic(X)=N^1(X)$ is finitely generated.
\item there are finitely many birational maps $f_i\colon X\dashrightarrow X_i$ for $1\leq i\leq k$, which are isomorphisms in codimension one, such that if $B$ is a mobile divisor then there is an index $1\leq i \leq k$ and a semiample divisor $B_i$ on $X_i$ such that $B=f_i^*B_i$ .\end{enumerate}}\end{defi}

The key point of this definition is that it allows one to run MMP on $X$ in a very easy and clear way. 
If $X$ is a Mori dream space then the pseudo-effective cone $\NNE^1(X)$ is divided into finitely many rational polyhedra, $R_1,\dots,R_m$,
\[\NNE^1(X)=\bigcup_{j=1}^mR_j\qquad.\]
The mobile cone is a union of $M_1,\dots,M_k$, some subset of the rational polyhedra $R_1,\dots,R_m$, and the birational maps $f_1,\dots,f_k$  defined in \ref{def-MDS} are precisely the maps $\varphi_{B_i}$ associated to a big mobile
divisor $B_i$ belonging to the interior of each polytope $M_i$. For details see \cite{Hu}~Proposition~1.11.

It was proved in \cite{BCHM}~Corollary~1.3.1 that any log Fano variety is a Mori dream space. In particular, a $dP_2$ fibration is a Mori dream space. The idea of defining techniques in this article is that we are trying to find $dP_2$ fibrations $X\subset\F$ whose decomposition of $\MMob(X)$ into $M_1,\dots,M_k$ coincides with the decomposition of $\MMob(\F)$ into such polytopes. In other words, $X$ is embedded into $\F$ and 
\[\Cox(X)=\Cox(\F)\slash(f=0)\qquad.\]

\begin{lm} Let $X\subset\F$ be a hypersurface of the rank two toric variety in~\ref{general-F} defined by a homogeneous polynomial of bi-degree $(\omega,4)$. If $X$ is a $dP_2$ fibration then $\sigma=\left< L,X\cap D_3\right>$ is a subcone of $\MMob(X)$.\end{lm}
\begin{proof} Similar to the proof of Proposition~\ref{generators-of-mob}~(\ref{mob-cone-F}) one can check that $\Bs|L|$ is empty and $\Bs|D_3|$ has no fixed component. Note that $\Bs|D_3|$ is included in the locus $(x_3=x_4=0)$, and this locus must have codimension strictly bigger than 1. Otherwise, if $(x_3=x_4=0)$ defines a divisor on $X$ then Proposition~\ref{ro3} implies that $X$ is not a $dP_2$ fibration.\end{proof}


\subsection{The Picard group}\label{Pic-section}
The aim in this section is to prove $\Pic(X)\cong\Z^2$ for a general $X$ in Table~\ref{table!dP2list}.

Let us first recall some technical tools that we use in the proof. This includes a version of the Lefschetz hyperplane theorem and a generalised Kodaira vanishing theorem.

\begin{theorem}\label{vanishing}\emph{[Generalised Kodaira vanishing, \cite{mori-kollar}~Theorem~2.70.]} Let $(X,\Delta)$ be a proper klt pair. Let $N$ be a $\Q$-Cartier Weil divisor on $X$ such that $N\equiv M+\Delta$, where $M$ is a nef and big $\Q$-Cartier $\Q$-divisor. Then $\H^i(X,\O_X(-N))=0$ for $i<\dim X$.\end{theorem}

\begin{rmk}\emph{Let $V$ and $W$ be algebraic varieties. Recall that any algebraic map $\pi\colon V\rightarrow W$ can be decomposed into finitely many varieties $V_i\subset V$ of varying dimension, on each of which $\pi$ restricts to a map with constant fibre dimension.}\end{rmk}

\begin{defi}\label{D(pi)}\emph{[\cite{stratified}~\S2.2] Define $D(\pi)$, the {\it measure of deviation} of $\pi: V\rightarrow W$, to be
\[D(\pi)=\sup_i\{\text{(the fibre dimension of $\pi$ in $V_i$) }-\text{(the codimension of $V_i$ in $V$)}\}\quad.\]}\end{defi}

\begin{theorem}\label{lefschetz}\emph{[Lefschetz hyperplane theorem, \cite{stratified}~\S2.2]} Let $\pi\colon V\rightarrow \C^N$ be a proper map of a purely $n$-dimensional (possibly singular) algebraic variety into complex affine space. Then $\H_i(V)=0$ for $i>n+D(\pi)$.\end{theorem}

\begin{lm}\label{u-homology} Let $X\subset\F$ be a hypersurface defined by
\[\left(\begin{array}{c}-e\\4\end{array}\right)\subset\left(\begin{array}{cccccc}
1&1&-\alpha_1&-\alpha_2&-\alpha_3&-\alpha_4\\0&0&\beta_1&\beta_2&\beta_3&\beta_4\end{array}\right)\qquad,\]
 where the variables are in order $u=v\prec~x_1\preceq~x_2\preceq~x_3\preceq~x_4$ and $\{\beta_1,\,\beta_2,\,\beta_3,\beta_4\}=\{1,\,1,\,1,\,2\}$. Suppose $\F_i$ and $X_i$ are birational models of $\F$ and $X$ obtained by small modifications as in Theorem~\ref{thm!BZ} and Definition~\ref{Xi}. Let $\U_i=\F_i-X_i$ be the complement of each $X_i$ in $\F_i$. Consider the point $x=(-e,4)\in\Z^2$ and recall from Proposition~\ref{generators-of-mob} that $\MMob(\F)$ is a cone in $\R^2=\Z^2\otimes\R$ with the same copy of $\Z^2$. If $X\in\Int(\MMob(\F))$, then $\H_5(\U_i)=\H_6(\U_i)=0$ for some $i$.\end{lm}
 
\begin{proof} Consider the map $\Phi_{|D|}\colon\F\rightarrow\P^N$ defined by the linear system of the divisor $D~=~4M-eL$ and assume $D\in\MMob(\F)$. By Proposition~\ref{generators-of-mob}, $\NNE^1(\F)$ has the following decomposition:
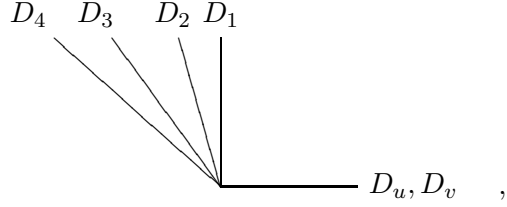
\[\xygraph{
!{(0,0) }="a"
!{(2,0) }*+{D_u,D_v}="b"
!{(0,1.8) }*+{D_1}="c"
!{(-0.5,1.8) }*+{D_2}="d"
!{(-1.3,1.8) }*+{D_3}="e"
!{(-2,1.8) }*+{D_4}="f"
"a"-"b" "a"-"c" "a"-"d" "a"-"e" "a"-"f"
} \quad, \]
where the rays are labelled by divisors that lie on them away from the origin.

From geometric invariant theory we have the following characterisation (possibly after taking a positive multiple of $D$):
\begin{enumerate}[(i)]
\item $\Phi_{|D|}$ is an embedding of $\F_i$ if $D\in\Int\left<D_i,D_{i+1}\right>$, where $D_i$ and $D_{i+1}$ do not lie on the same ray. 
\item $\Phi_{|D|}$ is a small contraction from $\F_i$ if $D=aD_i$ for some positive integer $a$ and $D_i\in\Int(\MMob(\F))$.
\item $\Phi_{|D|}$ is an extremal contraction of divisorial or fibre type otherwise.
 \end{enumerate}
 
 Suppose $D\in\Int(\MMob(\F))$; in particular it is in one of the cases $(i)$ or $(ii)$ above.\\
 
 Let $\U_i=\F_i-X_i$, where $i$ is the integer for which $(i)$ or $(ii)$ above is satisfied. Suppose $\varphi\colon\U_i\rightarrow\C^N$ be the restriction of $\Phi_{|D|}$ to $\U_i$. The map $\varphi$ is proper because $\Phi_{|D|}$ is a projective morphism and $X_i$ is the complete preimage of a hyperplane section of the target variety. Since this map contracts at most a 2-dimensional subspace of $\F_i$ and is isomorphism everywhere else, the codimension of every $V_j$ in Definition~\ref{D(pi)} is at least 2, while the fibre dimension is at most $2$. Hence $D(\varphi)\leq 0$ so by Theorem~\ref{lefschetz} we conclude that $\H_5(\U_i)=\H_6(\U_i)=0$. Note that $\dim_{\C}(\U_i)=4$ and $\dim_\R(\U_i)=8$.\end{proof}

\begin{cor}\label{H2(U)=0} $\H^2_c(\U_i)=\H^3_c(\U_i)=0$.\end{cor}
\begin{proof} This follows from Lemma~\ref{u-homology} and Poincar\'e duality .\end{proof}

\begin{lm}\label{H2F} Let $\F$ be the ambient toric variety of any family in Table~\ref{table!dP2list} except 1,2 and 3. Then $\H^2(\F_i)=\Z^2$ for all models $\F_i$ obtained by flips, flops or antiflips from $\F$.\end{lm} 

\begin{proof} From the short exact sequence
\[0\rightarrow\Z\rightarrow\O_\F\rightarrow\O^*_\F\rightarrow 0\]
one constructs the long exact sequence
\[\cdots\rightarrow\H^1(\F,\Z)\rightarrow\H^1(\F,\O_\F)\rightarrow\H^1(\F,\O_\F^*)\rightarrow\H^2(\F,\Z)\rightarrow\H^2(\F,\O_\F)\rightarrow\cdots\quad.\]

On the other hand, for any $\F$ in Families~4,\dots ,12 in Table~\ref{table!dP2list} there exists a birational model $\F_i$, obtained by some flips (flops or antiflips) for which $-K_{\F_i}$ is nef and big. Applying Theorem~\ref{vanishing} for the pair $(\F_i,0)$ and divisor $-K_{\F_i}$ gives $\H^j(\F_i,\O_{\F_i}(-K_{\F_i}))=0$ for all $j<4$. This vanishing together with Serre duality implies 
\[\H^1(\F_i,\O_{\F_i})=\H^2(\F_i,\O_{\F_i})=0\qquad.\]
The fact that $\F_i$ have rational singularities ensures that the vanishing above holds for all models $\F_i$. 

Of course $\Pic(\F_i)\cong\Z^2$ for all models $\F_i$ obtained by flips, flops or antiflips from $\F$. Using the fact that $\H^1(\F_i,\O^*_{\F_i})\cong\Pic(\F_i)$, the exact sequence above, together with the vanishing statements that we proved imply $\H^2(\F_i)\cong\Z^2$.
\end{proof}

\begin{prop}\label{H2=Z2} Let $X\subset\F$ be a hypersurface defined by $f\in\H^0(\F,D)$, where $D=4M-eL$ and $(-e,4)\in\Int(\MMob(\F))$. If $\F$ is the abient space of one of the families in Table~\ref{table!dP2list} except families 1,2 and 3, then $\H^2(X_i)\cong\Z^2$ for $X_i\subset\F_i$, where $\F_i$ is the model specified in Lemma~\ref{u-homology}.\end{prop}

\begin{proof} Together with Corollary~\ref{H2(U)=0}, the exact sequence
\[\cdots\rightarrow\H^2_c(\U_i)\rightarrow\H^2(\F_i)\rightarrow\H^2(X_i)\rightarrow\H^3_c(\U_i)\rightarrow\cdots\]
implies $\H^2(\F_i)\cong\H^2(X_i)$. The proof follows from Lemma~\ref{H2F}.
\end{proof}

\begin{lm}\label{Xvanishing} For a general $X$ in Table~\ref{table!dP2list}, $\H^1(X,\O_X)=\H^2(X,\O_X)=0$.\end{lm}

\begin{proof} For any such $X$ there exists a model $X_i$ obtained by some flips, flops or antiflips from $X$ such that $-K_{X_i}$ is nef and big on $X_i$. Considering the pair $(X_i,0)$, which is a klt pair as $X_i$ is terminal, and applying Theorem~\ref{vanishing} gives $\H^j(X_i,\O_X(-K_{X_i}))=0$ for all $j<3$. This together with Serre duality implies $\H^1(X_i,\O_{X_i})=\H^2(X_i,\O_{X_i})=0$. The rationality of singularities of $X_i$ allows one to lift this vanishing to all $X_k$. In particular, $\H^1(X,\O_X)=\H^2(X,\O_X)=0$.\end{proof}

\begin{theorem}\label{piX-general} Let $X\subset\F$ be a general $dP_2\slash\P^1$ in one of the families in Table~\ref{table!dP2list} then $\Pic(X)\cong\Z^2$.\end{theorem}
\begin{proof} Let $X$ be a general $dP_2\slash\P^1$ in one of the families of Table~\ref{table!dP2list} except families 1, 2 and 3. By Proposition~\ref{H2=Z2}, $\H^2(X_i)\cong\Z^2$ for some model $X_i$ obtained by some flips, flops or antiflips from $X$. On the other hand, Lemma~\ref{Xvanishing} implies $\H^1(X_i,\O_{X_i})=\H^2(X_i,\O_{X_i})=0$. Applying this to the exact sequence
\[\cdots\rightarrow\H^1(X_i,\Z)\rightarrow\H^1(X_i,\O_{X_i})\rightarrow\H^1(X_i,\O_{X_i}^*)\rightarrow\H^2(X_i,\Z)\rightarrow\H^2(X_i,\O_{X_i})\rightarrow\cdots\]
enables one to see $\H^1(X_i,\O_{X_i}^*)\cong\H^2(X_i,\Z)$; hence $\Pic(X_i)\cong\Z^2$. The fact that $X_i$ is isomorphic to $X$ in codimension 1 shows that $\Pic(X)\cong\Z^2$.\\

In order to finish the proof, we must show that $\Pic(X)\cong\Z^2$ for a general $X$ in families 1, 2 and 3. But we know that any such $X$ is obtained by a blow up of a Fano 3-fold with Picard rank 1, which completes the proof.\end{proof}

\section{Failing cases}\label{proof-main-theorem}

In this section, we show that any hypersurface $X\subset\F$ under the hypothesis of Theorem~\ref{th!dP2main}, which does not appear in the Table~\ref{table!dP2list} either is not a $dP_2$ fibration or does not provide an $\F$-Sarkisov link. 

Let us fix a general setting for $\F$ and $X$. Let $\F$ be the rank two toric variety with Cox ring $\Cox(\F)=\C[u,v,x_1,x_2,x_3,x_4]$ and irrelevant ideal $I=(u,v)\cap(x_1,\dots,x_4)$ with the action of $(\C^*)^2$ defined by 
\begin{equation}\label{general-F}
\left(\begin{array}{cccccc}
1&1&-a_1&-a_2&-a_3&-a_4\\
0&0&b_1&b_2&b_3&b_4
\end{array}\right)\qquad,\end{equation}
where $a_i$ are non-negative integers and $\{b_1,\dots,b_4\}=\{1,1,1,2\}$ such that the coordinate variables of $\Cox(\F)$ are in order $u=v\prec x_1\preceq x_2\preceq x_3\preceq x_4$. Let $X$ be a hypersurface of $\F$ defined by a homogeneous polynomial of bi-degree $(\omega,4)$ with respect to the action above. 
We sometimes switch these variable names to our favourite $u,v,x,y,z,t$ when we need to write explicit equations. Otherwise, we keep this notation, as it enables us to consider the order of variables without confusion about the position of the variable $t$ and having to divide into three types described at the beginning of Section~4.2. .

\subsection{Elimination process}
Here we provide the key tools to eliminate cases which do not occur in Table~\ref{table!dP2list}.

In the following lemma, we consider the coordinate variables of $\F$ to be $u,v,x,y,z,t$ and the variable $t$ corresponds to the coordinate, which has been acted by $(\lambda^{-\gamma},\mu^2)\in(\C^*)^2$.
\begin{lm}\label{esign} If $X$ is taken as a hypersurface in $\mathcal{F}$, it fails to be terminal if any of the following holds:
\begin{enumerate}
\item $\mathcal{F}$ is of type $(i)$, and $e>2c$.

\item $\mathcal{F}$ is of type $(ii)$, and $e>0$.
\item $\mathcal{F}$ is of type $(iii)$, and $e>2$.
\end{enumerate}
\end{lm} 
\begin{proof} In any of these cases, whenever $t$ appears in a term of $f$, it is multiplied by a nonconstant polynomial in $x,y,z$, which implies $\Gamma_t\subset X$. We recall that the curve $\Gamma_t$ is defined as $\Gamma_t=(x=y=z=0)\subset X$. Therefore $X$ has a line of singularity, but 3-fold terminal singularities are isolated by \cite{miles-canonical}.\end{proof} 
We are interested in cases that $\sigma=\MMob(X)$. In particular, these are the cases when the type $\III$ and $\IV$ 2-ray game of $X$ follows the one from $\F$. The following lemma helps us to eliminate cases when $X$ fails to follow such link at the beginning of the game.

\begin{theorem}\label{2ray}
Let $X\subset \F$ be defined as in \ref{general-F}. If $X$ is not obtained by one of the following, then either it is not a $dP_2$ fibration or the first step of its 2-ray game cannot be obtained by the restriction of the one from $\F$.
\begin{enumerate}[(i)]
\item $a_1=a_2=a_3=a_4=0$ and $\omega=1$.
\item $a_1=a_2=a_3=0$, $a_4=1$ and $\omega=0$.
\item $a_1=a_2=0$, $a_3a_4\neq 0$ and $\omega=0$.
\item $x_1\prec x_2,x_3,x_4$ and there is a monomial with only variables $x_1,x_2,x_3,x_4$ in the defining equation of $X$.
\end{enumerate} 
\end{theorem}
\begin{proof} Assume $x_1,x_2,x_3,x_4$ have equal ratio weight, i.e. $x_1=x_2=x_3=x_4$. Then there is no $\Psi_i$ and the 2-ray game of $\F$ is followed by a fibration to $\P(1,1,1,2)$. Without loss of generality we can assume this common weight is zero. In other words, by adding a multiple of the second row of the matrix $A$ to the first row we can assume $X\subset\F$ is defined by
\[\left(\begin{array}{c}
\omega\\4
\end{array}\right)
\subset\left(\begin{array}{cccccc}
1&1&0&0&0&0\\0&0&1&1&1&2\end{array}\right)\qquad.\]
If $\omega=0$, then $X\cong\P^1\times dP_2$. If we denote the generic fibre by $S$, then $\H^1(S,\mathcal{O}_S)=0$ together with Exercise~12.6 in Chapter~III \cite{Hart} implies that $\Pic(X)=\Pic(S)\times\Pic(\P^1)$. And hence $\rho_X>2$ and therefore $X$ is not a Mori fibre space.  
\noindent If $\omega=1$, then the equation of $X$ has the form $uf=vg$ for $f,g$ degree 4 homogeneous polynomials in $\P(1,1,1,2)$. It shows that $X$ is the blow up of $\P(1,1,1,2)$ along a curve defined by $(f=g=0)$. This was done by restricting $\Phi^\prime$ to $X$, which shows the 2-ray game of $X$ comes from $\F$. This case was given as Family~1 in Table~\ref{table!dP2list}.\\
If $\omega>1$, then $X$ is generically  an $\omega$-cover of $\P(1,1,1,2)$, which fails to be a $dP_2$ fibration.\\

To move onto the next case, suppose the ratio weight of $x_1,x_2,x_3$ is equal and normalised to zero and different from that of $x_4$. In other words, $x_1=x_2=x_3\prec x_4$ and $X\subset\F$ is defined by
\[\left(\begin{array}{c}
\omega\\4
\end{array}\right)
\subset\left(\begin{array}{cccccc}
1&1&0&0&0&-a\\0&0&b_1&b_2&b_3&b_4\end{array}\right)\qquad,\]
for a positive integer $a$.
In this case, the 2-ray game of $\F$ is followed by a divisorial contraction to $\displaystyle{\P=\Proj\bigoplus_k\Cox(\F)_{(0,k)}}$, with exceptional divisor $(x_4=0)$. If $\omega<0$, then $X$ is reducible and hence not a $dP_2$ fibration.\\
If $\omega=0$ and $a=1$, then $\varphi^\prime$ is a divisorial contraction from $X$, which is case (ii). This forms Family~2 and Family~3 in Table~\ref{table!dP2list}. The failure of case $\omega=0$ and $a>1$ is proved in Lemma~\ref{w=0-a>1} below.\\
The interesting case is when $\omega>0$. In this situation the image of restriction of the contraction on $\F$ to $X$ is a surface, hence this map does not define the 2-ray game of $X$. This means that $X$ does not have an $\F$-Sarkisov link. But when $b_4=\omega=a=1$, we show in Example~\ref{unproj} that $X$ is non-rigid. Note that this case does not appear in Table~\ref{table!dP2list} as the 2-ray game is given by a different ambient space. Apart from this special case, if $X$ forms a $dP_2$ fibration, we expect it to be non-rigid.\\

For part $(iii)$, assume $a_1=a_2=0$ and $x_1,x_2\prec x_3,x_4$. In this case, the 2-ray game of $\F$ is continued by an anti-flip (or flop), which contracts $\P^1\times\P^1$ to $\P^1$ and extracts a copy of $\P^1\times\P(a_3,a_4)$. If $\omega=0$, then the restriction of this operation to $X$ will be a finite number (2 or 4) of disjoint anti-flips (or flops) of type $(1,1,-a_3,-a_4)$. This is the case mentioned in $(iii)$.\\
If $\omega< 0$, then the Picard number of $X$ is bigger than two, which is proved in Proposition~\ref{ro3}. This shows that $X$ is not a $dP_2$ fibration.\\
If $\omega>0$, then the restriction of the ambient anti-flip (flop) defines an small contraction in one side and an isomorphism in the other side, which clearly does not read the 2-ray game of $X$.\\

Assume $x_1\prec x_2,x_3,x_4$. In this case the 2-ray game of $\F$ at the level of $\Psi_1$ can be read as a flip (flop or anti-flip) of type $(\alpha,\alpha,-\beta_1,-\beta_2,-\beta_3)$. It is obvious that this will restrict to a 3-fold flip (flop or anti-flip) on $X$ if  the extracted surface, $\P(\beta_1,\beta_2,\beta_3)$ with coordinate variables $x_2,x_3,x_4$, intersected with $X$ defines a curve. This will be valid only if this surface is not a subvariety of $X$. This means the defining polynomial of $X$ must have at least one monomial with only $x_i$ variables. Note that if a term of the form $x_1^k$ appears in the equation, $X$ will pass this step of the 2-ray game isomorphically and nothing contradicts our statements.
\end{proof}

\begin{ex}\label{unproj}\emph{Let $X\subset\F$ be defined in the usual way by}
\[\left(\begin{array}{c}1\\4\end{array}\right)\subset\left(\begin{array}{cccccc}
1&1&0&0&0&-a\\0&0&1&1&2&1\end{array}\right)\qquad,\]
\emph{where $a>0$ is an integer. It was shown in the proof of Theorem~\ref{2ray} that such $X$ does not have an $\F$-link. Here we show that $X$ can be embedded into another scroll $\F^\prime$ such that $X$ has an $\F^\prime$-Sarkisov link to another Mori fibre space.}

\emph{Let us fix the variables of $\F$ in order by $u,v,x,y,t,z$ as usual. The defining polynomial of $X$ is of the form $uf=vg$ for some bi-degree $(0,4)$ polynomials $f,g$. Now we apply unprojection operations of \cite{papa}. Let $s$ be  a rational function defined by 
\[s=\frac{f}{v}=\frac{g}{u}\]
with bi-degree $(-1,4)$. Then treat it as a variable in equations $us=g$ and $vs=f$. This enables us to embed $X$ into the scroll $\F^\prime$:
\[\left(\begin{array}{ccccccc}
1&1&0&0&0&-1&-a\\0&0&1&1&2&4&1\end{array}\right)\qquad,\]
where the variables in order are $u=v\prec x=y=t\prec s\prec z$. The variety $X$ is embedded into $\F^\prime$ as the complete intersection of two hypersurfaces $us=g$ and $vs=f$.}

\emph{$\F^\prime$ is a 5-fold toric variety of rank 2 whose 2-ray game starts by an anti-flip (or flop) of type $(1,1,-1,-a)$ over a surface $\P(1,1,2)$. Meaning, it contracts a copy of $\P^1\times\P(a,a,2)$ to $\P(1,1,1)$ in one side and extracts a copy of $\P(1,a)\times\P(1,1,2)$ in the other side. The restriction of this map to $X$ defines an anti-flip (or flop), consisting 2 disjoint anti-flip (or flop) of type $(1,1,-1,-a)$. Then it has a divisorial contraction to a codimension 2 Fano 3-fold of index one defined by $Y_{4,4}\subset\P(1,1,1,1,2,3)$.}\end{ex}

The key point in this example is that the $\sigma\subset\MMob(X)$ but they are not equal. However, as $-K_X$ is still in the pseudo-effective cone, we managed to find another embedding of $X$ for which $\MMob(X)$ is the restriction of that of the ambient space. This allowed us to read $-K_X\in\Int(\MMob(X))$.

Before stating the next lemma, we say a few words about the anticanonical classes of $\F$ and $X$. By Corollary~2.2.6 in~\cite{goto} the anticanonical divisor of $\F$ has bi-degree $(2-\sum a_i,\sum b_i)$. By adjunction we have
\[-K_X=(-K_{\F}-X)|_X\]
and hence the anticanonical divisor of $X$ has bi-degree $(2-\sum a_i-\omega,1)$.

\begin{lm}\label{lem!-k} Let $X$ be a hypersurface of $\mathcal{F}$, as in the assumption of Theorem~\ref{2ray}, satisfying conditions of Theorem~\ref{2ray} and Lemma~\ref{esign}, which has an $\F$-link. If $-K_X\sim mD_3-nD_u$ for a positive integer $m$ and a non-negative integer $n$, then the last morphism of the 2-ray game of $X$ is not an extremal contraction.\end{lm}

\begin{proof}
The proof is given case by case, depending on the ratio weights of the variables. In each case we find a curve inside the exceptional locus of $\varphi^\prime$, which has positive intersection against the anticanonical class. This shows that the last morphism of the 2-ray game is not an extremal contraction. 
\begin{enumerate}[\text{Case} I]
\item $x_2\prec x_3\preceq x_4$\\
Let $C=(x_1=x_4=f=0)\subset\Exc(\varphi^\prime)$, where $f$ is the defining polynomial of $X$. Note that the irrelevant ideal of the domain variety of $\varphi^\prime$ is defined by $(u,v,x_1,x_2)\cap(x_3,x_4)$. Therefore $D_3\cdot C=0$, which implies
\[-K\cdot C=0-nD_u\cdot(x_1=x_4=f=0)\leq 0\]
\item $x_1\prec x_2=x_3\preceq x_4$\\
Let $C=(x_2=x_4=f=0)$. As the irrelevant ideal in this case is $(u,v,x_1)\cap(x_2,x_3,x_4)$, similar argument shows
\[-K\cdot C=0-nD_u\cdot(x_2=x_4=f=0)\leq 0\]
\item $x_1=x_2=x_3\prec x_4$\\
The irrelevant ideal in this case is $(u,v)\cap(x_1,x_2,x_3,x_4)$. Without loss of generality we can assume that $X$ is defined by
\[\left(\begin{array}{c}
\omega\\4
\end{array}\right)\subset\left(\begin{array}{cccccc}
1&1&0&0&0&-a\\
0&0&b_1&b_2&b_3&b_4
\end{array}\right)\quad,\]
where $a$ is a positive integer. Theorem~\ref{2ray} together with Lemma~\ref{w=0-a>1} implies $\omega=0$ and $a=1$.
\end{enumerate}
\end{proof}
\begin{rmk}\emph{Note that Lemma~\ref{lem!-k} implies that in order to have an $\F$-link from $X$, it is necessary for the ratio weight of $-K_X$ to be strictly less than that of the coordinate variable $x_3$. This is simply saying that $-K_X\in\Int(\MMob(X))$.}\end{rmk}

\begin{lm}\label{w=0-a>1} Let $X\subset\F$ be defined by
\[\left(\begin{array}{c}0\\4\end{array}\right)\subset\left(\begin{array}{cccccc}
1&1&0&0&0&-a\\0&0&1&1&2&1\end{array}\right)\qquad,\]
with variables in order $u=v\prec x=y=t\prec z$ with $a\in\Z$, $a\geq 1$. If the integer $a$ is strictly bigger than 1, then the image of the last morphism of the $2$-ray game of $X$ is not terminal.
\end{lm}
\begin{proof}
If $a>1$, then the image of $\F$ under the last morphism of its $2$-ray game is defined by the quotient of $\P(1,1,1,1,2)$ by the action of $\frac{1}{a}(1,1,0,0,0)$. In particular, this variety has a singular locus of dimension $2$. Hence the image of $X$ under this map has non-isolated singularities (along a curve) and therefore is not terminal.

\end{proof}

\begin{prop}\label{ro3} Let $X\subset\F$ be defined as before. If $D=(x_3=x_4=0)\subset X$ forms a divisor on $X$, i.e. if  the defining polynomial of $X$ is of the form $x_3f=x_4g$, then $\rho_X$, the Picard number of $X$, is at least 3.\end{prop} 

\begin{proof} As in the assumption, let the defining polynomial of $X$ be $x_3f=x_4g$ for non-constant polynomials $f,g$. Let $M\sim(x_1=0)$ and $L\sim(u=0)$ be two other divisors on $X$. We show that $D$, $M$ and $L$ are linearly independent and hence $\Pic(X)$ has at least three generators. To do so, we find three curves inside $X$ and compute their intersections with these divisors. These number form a $3\times 3$ matrix. If the rank of this matrix is bigger than 3, we have shown that these divisors are linearly independent.

Consider three curves $C_1,C_2,C_3\subset X$ defined by
\[C_1=(u=x_3=x_4=0)\qquad C_2=(x_1=x_3=x+4=0)\qquad C_3=((v=x_2=0)\cap X)\]

Computing intersection numbers gives:
\[\left(\begin{array}{ccc}
L\cdot C_1&L\cdot C_2&L\cdot C_3\\
M\cdot C_1&M\cdot C_2&M\cdot C_3\\
D\cdot C_1&D\cdot C_2&D\cdot C_3
\end{array}\right)=\left(\begin{array}{ccc}
0&1&0\\1&*&0\\ *&*&1
\end{array}\right)\quad,\]
where $*$ denotes some numbers that we have no interest in computing them.
Which shows that this matrix has full rank and hence $\rho_X>2$.\end{proof} 
 A typical example of a variety concerned in Proposition~\ref{ro3} has following shape:
\[X\in\left(\begin{array}{c}-1\\4\end{array}\right)\subset\left(\begin{array}{cccccc}
1&1&0&0&-1&-2\\0&0&1&1&1&2\end{array}\right)\]

Before we start the next section let us recall that $\F$ is said to be of type $(i)$, $(ii)$ or $(iii)$ if the corresponding action of $(\C^*)^2$ has the following representations. Note that an easy argument shows that any $\F$ considered in this article has a unique representation in one of these types. 
\[\begin{array}{llllll}
\vspace{0.4cm}
(i)&\qquad& A=\left(\begin{array}{cccccc}
1&1&0&-a&-b&-c\\
0&0&1&1&1&2\end{array}\right)&\qquad& 0<c\,,\,0\le a\le b&\\
\vspace{0.4cm}
(ii)&\qquad& A=\left(\begin{array}{cccccc}
1&1&-a&-b&-c&0\\
0&0&1&1&1&2\end{array}\right)&\qquad &\hspace{0.42cm} 0\le a\le b\le c&\\
(iii)&\qquad& A=\left(\begin{array}{cccccc}
1&1&-a&-b&-c&-1\\
0&0&1&1&1&2\end{array}\right)&\qquad& \hspace{0.42cm} 0<a\leq b\leq c&,\end{array}\]
where $a$, $b$ and $c$ are non-negative integers and the variables are $u,v,x,y,z,t$. The conditions on the order of $a,b,c$ imply that in all cases above the variables $x,y,z$ are ordered with $x\preceq y\preceq z$. And  if $\F$ is of type $(ii)$ or $(iii)$, then $t\preceq x$. 

Table~\ref{table!numerology} below gathers some computations of the anti-canonical class of $\F$ and $X$, which we use later.

\begin{table}[ht]\small
\[\begin{array}{c||c|c|c}
&\text{Type }(i)&\text{Type }(ii)&\text{Type }(iii)\\
\hline
\hline
-K_{\mathcal{F}}&(2-a-b-c)L+4M&(2-a-b-c)L+4M&(1-a-b-c)L+4M\\
\hline
-K_X&(2+e-a-b-c)L+M&(2+e-a-b-c)L+M&(1+e-a-b-c)L+M\\
\end{array}\]
\caption{Anticanonical classes of $\F$ and $X$\label{table!numerology}}
\end{table}

In the next two subsection, we explicitly analyse cases which do not occur in Table~\ref{table!dP2list} and give arguments why each of them fails. Our arguments are based on the materials provided in this part, namely Lemma~\ref{esign}, Theorem~\ref{2ray}, Lemma~\ref{lem!-k} and Proposition~\ref{ro3}. 

\subsection{Hypersurfaces in scrolls of Type $(ii)$ or $(iii)$}

\begin{prop}\label{3fail} If $\mathcal{F}$ is of type $(iii)$, then $X$ does not have a link to any other Mori fibre space except for $e=2, a=b=c=1$.\end{prop}
\begin{proof} If $e=2$, then Lemma \ref{lem!-k} implies $a+c<3$, and that means $a=b=c=1$. Under these numerical conditions a general $X$ passes the first step of the $2$-ray game isomorphically and then maps to $\mathbb{P}^2$ with conic fibres. This forms Family~6 in Table~\ref{table!dP2list}.\\
The case $e>2$ is not concerned, due to Lemma~\ref{esign}.  For $e<2$, Lemma~\ref{lem!-k} does the elimination.\end{proof}

\begin{prop}\label{2fail}
Suppose $\F$ is of type $(ii)$, and consider its 2-ray game of Type~\III\ or \IV.
Exactly one of the following cases occurs:
\begin{enumerate}
\item
$X$ does not have an $\F$-link, or
\item
$X$ does have an $\F$-link but it does not lead to an $\F$-Sarkisov link on $X$, or
\item
$X$ follows the 2-ray game of $\F$ to a Sarkisov link, and we are in one of the cases
\begin{enumerate}[(A)]
\item\label{eq!e=01} $e=a=0$, $b=c=1$,
\item\label{eq!e=02} $e=a=b=0$, $c=1$,
\item\label{eq!e=-1} $e=-1$, $a=b=c=0$.
\end{enumerate}
\end{enumerate}
\end{prop}
\begin{proof}
Suppose the given 2-ray game on $\F$ does restrict to a Sarkisov link on $X$.
In particular, $X$ has terminal singularities, so $e \le 0$ by Lemma~\ref{esign}.
If $e<0$, Lemma~\ref{esign} requires $S_k(u,v)t^2\in f$, where $S_k$ is a general polynomial with variables $u,v$ of degree $-e=k>0$. The numerology presented in Table~\ref{table!numerology}, shows that $-K_X\sim (2-k-a-b-c)L+M$. This, together with Lemma~\ref{lem!-k}, gives the inequality $k+a+c<2$. But this can be satisfied only if $k=1$ and $a=b=c=0$, which is the case~\eqref{eq!e=-1}.\\
In the case $e=0$, a similar argument using the result of Lemma \ref{lem!-k} forces $a+c<2$, and this leads immediately to cases~(\ref{eq!e=01},\ref{eq!e=02}) or $e=a=b=c=0$. but this case gets eliminated by Theorem~\ref{2ray}.\end{proof}

In fact, all solutions (\ref{eq!e=01}--\ref{eq!e=-1}) provide Sarkisov links when $X$ is general; these are respectively families No. 5, 2 and 1 in Table~\ref{table!dP2list}.\\

\subsection{Families embedded in Type $(i)$ scrolls}
Let us recall that the variable with ratio weight zero is fixed to be $x$ throughout this part.\\
The following lemma forces strong restrictions on $f$, the defining polynomial of $X$. It uses the condition on the singularities of $X$.

\begin{lm}\label{lin} Let $X\subset\F$ be a hypersurface of $\F$ of a Type~$(i)$, defined by the polynomial $f$ as
\[\left(\begin{array}{c}-e\\4\end{array}\right)\subset\left(\begin{array}{cccccc}
1&1&0&-a&-b&-c\\
0&0&1&1&1&2\end{array}\right)\qquad,\]
where $a,b,c>0$. If there is no term of the form $S_d(u,v)x^kl(y,z,t)$ in the equation of $f$, then $X$ is not terminal, where $l$ is either a linear form on $y,z,t$ or is a constant.\end{lm}
\begin{proof} By Theorem~\ref{2ray}, $f$ must include at least a monomial with no $u$ or $v$ in it. This already means $e\geq 0$. Let $\Gamma$ be the curve defined by $(y=z=t=0)$. If $e=0$, then $x^4\in f$ and there is nothing to prove. If $e>0$, then $\Gamma\subset X$ and in fact by easy computations one could see that $\Gamma\subset\Bs|D|$. If there is no term of the form $S_d(u,v)x^kl(y,z,t)$ in $f$, then $X$ is singular along $\Gamma$. In particular, the singular locus of $X$ is not isolated and hence $X$ cannot be terminal.\end{proof}

If $a,b,c$ are all nonzero, then by Theorem~\ref{2ray} $f$ must include at least one pure monomial in the $x,y,z,t$ variables. But this monomial cannot be $x^4$, as if otherwise holds, then Lemma~\ref{lem!-k} implies $a+c<2$ which cannot be satisfied for any pair of positive integers $a$ and $c$. Hence $abc\neq 0$ implies $e\neq 0$.

On the other hand, if one of $a,b,c$ is zero, then Proposition~\ref{ro3} implies $e=0$. If only two of $a,b,c$ is zero, then irreducibility of $X$ forces $e=0$. The case $a=b=c=0$ has been considered in Theorem~\ref{th!dP2main}.\\

\noindent The following families have already been studied in Theorem~\ref{th!dP2main}.
\[X\in\left(\begin{array}{c}0\\4\end{array}\right)\subset\left(\begin{array}{cccccc}
1&1&0&0&-1&-1\\0&0&1&2&1&1\end{array}\right)\]
\[X\in\left(\begin{array}{c}0\\4\end{array}\right)\subset\left(\begin{array}{cccccc}
1&1&0&0&0&-1\\0&0&1&1&2&1\end{array}\right)\]
\[X\in\left(\begin{array}{c}0\\4\end{array}\right)\subset\left(\begin{array}{cccccc}
1&1&0&0&-1&-1\\0&0&1&1&2&1\end{array}\right)\]

Now we consider the families with $e>0$. We will specify each family by a sequence of positive integers correspond to $(a,b,c;e)$ which represent the following:
\[X\in\left(\begin{array}{c}-e\\4\end{array}\right)\subset\left(\begin{array}{cccccc}
1&1&0&-a&-b&-c\\0&0&1&1&1&2\end{array}\right)\]
Note that the columns of the action matrix of $\mathcal{F}$ are not necessarily in order. But the $2$-ray game is played each time after considering the appropriate order.\\
We also introduce two numbers $n$ and $\kappa$, which will simplify our notation, by
\[n=a+b+c,\qquad \kappa=2+e-a-b-c\qquad.\]

 Note that the number $\kappa$ is associated to the degree of the anticanonical class of $X$ and determines it uniquely as $-K_X\sim \kappa L+M$. Let us recall that $L$ is the divisor linearly equivalent to $(u=0)$ and $M$ is the one equivalent to $(x=0)$.\\
We will be considering every $X$ defined by $(a,b,c;e)$ by varying $n\in\mathbb{N}$ and spot families which link to a different Mori fibre space. The cases $n=0, 1, 2$ have already been analysed.\\

\noindent$\bullet\quad n=3$

\noindent The only option for $n=3$ is when $a=b=c=1$. By Lemma~\ref{esign} $e\leq c$, which can only be satisfied by $e=1,2$. The analysis of the case $(1,1,1;1)$ is the Family~7 in Table~\ref{table!dP2list}.

A general $X$ defined by $(1,1,1;2)$ is not terminal as it does not not satisfy conditions of Lemma~\ref{lin}.\\

\noindent $\bullet\quad n=4$

This case has only two possibilities: $(1,1,2;e)$ and $(1,2,1;e)$. By Lemma \ref{lin} we must have $e\leq 2$. If $e<2$, for both cases $X$ fails to satisfy Lemma~\ref{lem!-k}. Remaining cases provide $\F$-Sarkisov links to other Mori fibre spaces. These are Families~8 and 9 in Tables~\ref{table!dP2list}.\\

\noindent $\bullet\quad n=5$

Different partitions of $5$ allow us to have $(1,1,3;e)$, $(1,3,1;e)$, $(1,2,2;e)$ or $(2,2,1;e)$. For the first two cases, $e$ cannot be less than 3 as otherwise it fails to fulfil the criteria of Lemma~\ref{lem!-k}. It also cannot be more than 3 because of the condition imposed by Lemma~\ref{lin}. A similar argument for the other two cases bounds $e$ to be equal to 2.\\
However, $(1,3,1;3)$ does not have Picard number two by Proposition~\ref{ro3}. $(1,2,2;2)$ also fails to satisfy Lemma~\ref{lem!-k} condition. The only remaining cases win to provide $\F$-Sarkisov links form Families~10 and 11 in Table~\ref{table!dP2list}.\\

\noindent $\bullet\quad n=6$

Possible partitions of 6 give three candidates $(1,1,4;e)$ , $(1,2,3;e)$ , $(2,2,2;e)$. Applying numerical conditions imposed by Lemma \ref{lem!-k}, Lemma~\ref{lin} and Proposition~\ref{ro3}, and running the elimination process, we are left with the $(1,1,4;4)$ and $(1,2,3;3)$. In Lemma~\ref{1,1,4,4}, a reason for failure of $(1,1,4;4)$ is given. The case $(1,2,3;3)$ is precisely the Family~12 in Table~\ref{table!dP2list}.

\begin{lm}\label{1,1,4,4} Let $X\subset\F$ be defined by
\[\left(\begin{array}{c}-4\\4\end{array}\right)\subset\left(\begin{array}{cccccc}
1&1&0&-1&-1&-4\\0&0&1&1&1&2\end{array}\right)\qquad,\]
with variables $u,v,x,y,z,t$ and equation $f$. Then a general $X$ has Picard number strictly bigger than 2.\end{lm}
\begin{proof} The proof here is the standard method used in Proposition~\ref{ro3}. The only difference here is that instead of working with $X$ we consider $X_1$, obtained by flopping a curve in $X$. Considering the 2-ray game of $X$ restricted from that of $\F$, there is an Atiyah flop on $X$ because we have a term $x^2t\in f$, which allows one to eliminate $t$ in a neighbourhood of $\Gamma_x$. As $X_1$ is obtained by flopping a curve in $X$, they have isomorphic Picard groups. Hence $\rho_{X_1}>2$ implies $\rho_X>2$.

In order to finish the proof, we need to show that there are at least three divisors on $X_1$, which are linearly independent. We specify three divisors below and then conclude by proving they have non-linearly dependent intersections with three specific curves inside $X_1$.
After a suitable change of coordinates we can assume $f=yz(y-z)(y-\lambda z)+t(x^2+g)$ (for some fixed cross ratio $\lambda$), where $g$ is a polynomial of bi-degree $(0,2)$. Setting $t=0$ in $X_1$ leaves 4 divisors above the four roots $0, 1, \lambda, \infty$ of the quartic in $y,z$, each of them a divisor in $X_1$ isomorphic to $\P^2_{u:v:x}$ . Let $D$ be the divisor defined by $(y=1, z=t=0)$ and suppose $L\sim(u=0)$ and $M\sim(x=0)$ are two other divisors of $X_1$. We show that these divisors are linearly independent.

Define three curves on $X_1$ by
\[C_1=(v=x=f=0),\quad C_2=(v=z=f=0),\quad C_3=(x=y=t=0)\qquad.\]

Computing the intersections leads to
\[\left(\begin{array}{ccc}
C_1.L&C_1.M&C_1.D\\
C_2.L&C_2.M&C_2.D\\
C_3.L&C_3.M&C_3.D\end{array}\right)=
\left(\begin{array}{ccc}
0&2&1\\
1&2&0\\
1&1&0
\end{array}\right)\quad.\]
This matrix has full rank and this completes the proof.
\end{proof}

\noindent $\bullet\quad n=7$

Considering different partitions of $7$ and applying the numerical elimination process as before, it turns out that there is only one family of three-folds for which a general member is not birationally rigid, which is $(1,2,4;4)$. This forms Family~13 in Table~\ref{table!dP2list}.

The following lemma shows that we only need to consider cases where $n\leq 7$.

\begin{lm} Any $X$ with $n>7$ does not link to any other Mori fibre space by an $\F$-link.\end{lm}
\begin{proof} Let $X$ be defined by
\[\left(\begin{array}{c}-e\\4\end{array}\right)\subset\left(\begin{array}{cccccc}
1&1&0&-\alpha_1&-\alpha_2&-\alpha_3\\0&0&1&\beta_1&\beta_2&\beta_3\end{array}\right)\qquad,\]
 where $\{\beta_1,\,\beta_2,\,\beta_3\}=\{1,\,1,\,2\}$ and variables are in order $u=v\prec x\preceq x_1\preceq x_2\preceq x_3$. Lemma~\ref{lin} implies $e\in\{\alpha_i-m\, |\, 1\leq i\leq 3\, ,\, m=0,1\}$. By adjunction $\displaystyle{-K_X\sim (2-m+\alpha_i-\Sigma\alpha_j)L+M}$. To fulfil $-K_X\in\Int(\MMob(X))$, the requirement of Lemma~\ref{lem!-k}, we must have 
 \[m+\alpha_1+\alpha_2+\alpha_3-\alpha_i<2+\frac{\alpha_2}{\beta_2}\qquad.\]
Proposition~\ref{ro3}, together with Lemma~\ref{lem!-k} and Theorem~\ref{2ray}, shows that this inequality  has no solution for any choice of $m$ and $i$. \end{proof}

\section{Cubic surface fibrations over $\P^2$}
In this section we consider a similar construction and provide a list of non-rigid families for cubic surface fibrations over $\P^2$. The arguments are very similar and we do not repeat them for this case.
\begin{defi}\label{4fold dP3} \emph{A 4-fold cubic fibration over $\P^2$ is a normal, irreducible, projective, complex variety $X$ such that
\begin{enumerate}[(a)]
    \item $X$ is $\Q$-factorial with at worst terminal singularities,
    \item $\Pic X \cong \mathbb{Z}^2$,
    \item there exists an extremal morphism of fibre type $\varphi \colon X \rightarrow \mathbb{P}^2$, and
    \item the generic fibre of $\varphi$ is a degree 3 del Pezzo surface.
\end{enumerate} 
We denote this by $dP_3\slash\P^2$.}\end{defi}

Let $\F$ be a weighted bundle over $\P^2$ defined by
\begin{enumerate}[(i)]
\item $\Cox(\F)=\C[u,v,w,x,y,z,t]$,
\item $I_{\F}=(u,v,w)\cap(x,y,z,t)$,
\item $(\C^*)^2$ action defined by
\begin{equation}\left(\begin{array}{rrrrrrr}
1&1&1&\alpha&\beta&\gamma&\delta\\
0&0&0&1&1&1&1
\end{array}\right)\quad,\end{equation}
for $\alpha,\beta,\gamma,\delta\in\Z$.
\end{enumerate}

\subsection{Construction as hypersurfaces}\label{constructiondP3} Without loss of generality we can assume that matrix above is of the form
\begin{equation}\label{matrix of 4-fold dP3}\left(\begin{array}{rrrrrrr}
1&1&1&0&-a&-b&-c\\
0&0&0&1&1&1&1
\end{array}\right),
\end{equation}
where $a\leq b\leq c$ are non-negative integers. In particular, the variables are in the order $u=v\prec x\preceq y\preceq z\preceq t$.

We denote the basis of $\Pic(\mathbb{F})$ by $L\,,\,M$, with sections $u\in H^0(\F,L)$ and $x\in H^0(\F,M)$.\\ 
Let $D\in|4M+dL|$ be a divisor in $\F$ for $d\in\Z$ and suppose $X\subset\F$ is a hypersurface defined by $X=(f = 0)\subset\F$ for a general $f\in\O_{\F}(D)$. The aim is to study the birational geometry of those $X$ specified by $(a,b,c;d)$, which satisfy the conditions of Definition~\ref{4fold dP3}.

\subsection{$dP_3\slash\P^2$ models}

Here we find those $(a,b,c;d)$ for which the 3-fold $X$ forms a degree 3 del Pezzo surface fibration over $\P^2$, as in Definition~\ref{4fold dP3}.

\begin{lm}\label{d>0} Let $X\subset\F$ be a general hypersurface defined as in~\ref{constructiondP3} by sequence of integers $(a,b,c;d)$, where $0\leq a\leq b\leq c$ and $d>0$. Then a general $X$ is a $dP_3\slash\P^2$.\end{lm}
\begin{proof} If $d>0$, then the defining polynomial of $X$ is of the form $f=uf_1+vf_2+wf_3$ for some polynomials $f_i$ with bidegree $(d-1,3)$. It implies that the base locus of the linear system $|3M+dL|$ is empty and hence by the Bertini theorem $X$ is smooth. By Theorem~\ref{pic4fold} below, $\Pic(X)\cong\Z^2$ and hence $X$ is a $dP_3\slash\P^2$.\end{proof}

\begin{lm}\label{d=0} Let $X\subset\F$ be defined by $(a,b,c;0)$ as before. Then $X$ forms a $dP_3\slash\P^2$ for any triple $(a,b,c)$ except for $a=b=c=0$.\end{lm}
\begin{proof} It is easy to check that for any $(a,b,c)$, the base locus of $|3M|$ is empty and therefore $X$ is smooth. If $a=b=c=0$, then the Picard number of $X$ is strictly bigger than $2$. By Theorem~\ref{pic4fold} $\Pic(X)\cong\Z^2$ for all other cases.\end{proof}

\begin{lm}\label{d<0} Let $X\subset\F$ be a hypersurface defined by $(a,b,c;d)$ as in~\ref{constructiondP3}, where $0\leq a\leq b\leq c$ and $d<0$. Then $X$ is a $dP_3\slash\P^2$ if 
\begin{enumerate}[(i)]
\item the defining polynomial of $X$ includes a monomial of the form $g_k(u,v,w)x^2L(y,z,t)$, where $g_k$ is a homogeneous polynomial in variables $u,v,w$ of degree $k\geq 0$ and $L$ is a linear form in $y,z,t$, and
\item one of the following holds
\[d\leq 3a\leq 3b\quad\text{or}\quad d<3a\leq 3b\qquad.\]
\end{enumerate}\end{lm}

\begin{proof} If $a=b=c=0$, then $|3M+dL|$ has no sections. If $a=b=0$ and $c>0$, then $f=t.g$, hence $X$ is reducible. If only $a=0$ and $bc\neq 0$, then a similar argument to the one in Proposition~\ref{ro3} shows that $\rho_X>2$.

Let $0<a\leq b\leq c$ and suppose one of the $d\leq 3a\leq 3b$ or $d<3a\leq 3b$ holds. Then Theorem~\ref{pic4fold} implies that $\Pic(X)\cong\Z^2$. If $d=3a=3b$, then by a similar argument to Lemma~\ref{1,1,4,4}, $\rho_X>2$ and hence $X$ is not a $dP_3\slash\P^2$.

Now suppose $X$ is defined such that $0<a\leq b\leq c$. If the polynomial $f$ has no term of type $g_k(u,v,w)x^2L(y,z,t)$, then a generic point on the surface $S=(y=z=t=0)\subset X$ has multiplicity at least 2. Therefore $X$ is singular along a 2-dimensional space. Therefore $X$ is not terminal. If $f$ has such a term, then it is either smooth or it is singular only at finitely many points or along a line. \end{proof}

Combining Lemma~\ref{d>0}, Lemma~\ref{d=0} and Lemma~\ref{d<0} enables us to give the following characteristic theorem. 

\begin{theorem}\label{dP3/p2} Let $X\subset\F$ be a general hypersurface defined by $(a,b,c;d)$. Then one of the following holds:
\begin{enumerate}[(1)]
\item If $d>0$, then $X$ is non-singular and satisfies conditions stated in Definition~\ref{4fold dP3} .
\item If $d=0$, then $X$ is a $dP_3$ fibration by Definition~\ref{4fold dP3} for any triple $(a,b,c)$ except for $a=b=0$, $c>1$.
\item $d<0$ and
\begin{enumerate}[(a)]
\item $3c<-d$, $|4M+dL|$ has no sections.
\item $3a\leq 3b< -d\leq 3c$ and $X$ is reducible, hence not a $dP_3$ fibration.
\item $3a<-d\leq 3b\leq 3c$ and $X$ has Picard number $\rho_X>2$, hence does not satisfy conditions of a $dP_3$ fibration.
\item $-d\leq 3a$. In this case, $X$ is a $dP_3$ fibration over $\P^2$ only if the equation of $f$ has a term of the form $g_k(u,v,w)x^2L(y,z,t)$ in it, where $g_k$ is a homogeneous polynomial in variables $u,v,w$ of degree $k\geq 0$ and $L$ is linear.
\end{enumerate}
\end{enumerate}
\end{theorem}

\subsection{$dP_3\slash\P^2$ as Mori dream spaces}
In what follows we show that unlike dimension 3, all $dP_3$ fibrations constructed above have a 2-ray game which is the restriction of that of the ambient space we consider.  The idea is based on the following lemma of Kawamata, Matsuda and Matsuki.

\begin{lm}\label{kawamata-flip}\emph{(\cite{kawamata-flip}~Lemma~5.1.17)} If $\psi\colon X^-\rightarrow X^+$ is a flip (flop or antiflip) with exceptional loci $E^-\subset X^-$ and $E^+\subset X^+$, then the pair $(\dim E^-,\dim E^+)$ is exactly one of the pairs
\[(2,1)\qquad (2,2)\qquad (1,2)\qquad.\]\end{lm}

\begin{theorem}\label{2ray4fold} Let $X\subset\F$ be a cubic fibration over $\P^2$ obtained from one of the cases in Theorem~\ref{dP3/p2}. Then the Type~\III\ or \IV\ 2-ray game of $\F$ induces the game on $X$.\end{theorem}

\begin{proof} We prove the theorem case by case on the sign of $d$ and we show that in each case the conditions on the dimension of contracted loci by Lemma~\ref{kawamata-flip} are satisfied. 

Let $d>0$. If $a>0$, then the 2-ray game of $\F$ is continued be a flip which restricts to $X$ with dimension pair $(1,2)$. For $a=0$ and $b>0$, the situation is $(2,1)$ and for $a=b=0$ the game finishes by a divisorial contraction or a fibration; Which is fine as far as the 2-ray game of $X$ is concerned.

For $d=0$, If $a>0$ then the first step of the game of $\F$ induces an isomorphism on $X$ and the second step is of type $(2,1)$, divisorial contraction or fibration, respectively in cases $a,b$, $a=b<c$ and $a=b=c$.

If $a=0$, then the game continues with a $(2,1)$ or divisorial contraction or a fibration exactly as the previous case.

Let $d<0$. If $a>0$ then the 2-ray game of $\F$ restricts to $X$ by a $(2,1)$ or $(2,2)$.\end{proof} 

\begin{cor}\label{mobdP3} $X$ is a Mori dream space with $\Cox(X)=\Cox(\F)\slash(f=0)$. In particular $\MMob(X)$ is generated by $L$ and $D_z=(z=0)$.
\end{cor}

\subsection{Nonrigid families}

The following arguments eliminate cases that are not going to have an $\F$-link to another Mori fibre space. As a result a list of nonrigid families through their Type~$\III$ or $\IV$ Sarkisov links is given.

\begin{theorem}\label{-kdP3} If $-K_X\notin\Int(\MMob(X))$, then the last map of the 2-ray game of $X$ is not extremal.
\end{theorem}
\begin{proof} This proof is similar to that of Lemma~\ref{lem!-k}.
\end{proof}

\begin{lm}\label{ak} If $d<0$, then $a+k\leq 2$.
\end{lm}
\begin{proof} Using the adjunction formula, one can compute the anticanonical divisor of $X$ as $-K_X\sim(3+n-a-b-c)L+M$. Theorem~\ref{-kdP3} results in $-K_X\in\Int(\MMob(X))$, which holds if and only if $a+b+c-3-d<b$. This implies $a+c<3+d$.

On the other hand, from Theorem~\ref{dP3/p2} we have $d\leq c-k$. These two inequalities show that $a+k\leq 2$.
\end{proof}

\begin{cor}\label{c7} $c<7$.\end{cor}
\begin{proof} Theorem~\ref{-kdP3} implies $a+c<3-d$. On the other hand, Theorem~\ref{2ray4fold} requires $-d<c$. One can easily check the inequality using these together with Lemma~\ref{ak}.\end{proof}

The inequalities above provide upper limits for $(a,b,c)$. Using these and other information provided in this section one can prove that Theorem~\ref{dP3!list} below has the complete list.

\begin{theorem}\label{pic4fold} Let $X\subset\F$ be a general $dP_3\slash\P^2$ as before. If $X\in\Int(\MMob(\F))$, then $\Pic(X)\cong\Z^2$.\end{theorem}
\begin{proof} One can apply same method as in proof of Theorem~\ref{piX-general} to obtain this result. Note that the proof in this case is much easier as $\F$ and $X$ are smooth.\end{proof}

\begin{theorem}\label{list}\label{dP3!list}
 Consider a general hypersurface $X\subset \F$ with 
\[\left(\begin{array}{c}d\\3\end{array}\right)\subset\left(\begin{array}{cccccc}
1&1&0&-a&-a&-c\\0&0&1&1&1&1\end{array}\right)\quad,\]
where $0\le a\le b\le c$. 
If the the Type~\III\ or \IV\ 2-ray game of $X$ leads to another Mori fibre space, then the weights $(a,b,c;d)$ are among those appearing in the left-hand column of Table~\ref{table!dP3list1} and Table~\ref{table!dP3list2}.

The Sarkisov links generated in this way are described in the remaining columns of Tables~\ref{table!dP3list1} and~\ref{table!dP3list2}.\end{theorem}

\begin{landscape}

\begin{table}[ht]
\vspace{1.1cm}
\[\begin{array}{cc||c|c|c|c}
\text{No.}&(a,b,c;d)&\psi_1&\psi_2&\varphi^\prime&\text{new model}\\
\hline
\hline
1&(0,1,1;1)&\text{flip}&\text{n/a}&\text{fibration}&(Y_4\subset\P^4)\slash\P^1\\
\hline
2&(0,0,1;1)&\text{n/a}&\text{n/a}&\text{contraction}&\text{Fano }Y_4\subset\P^5\\
\hline
3&(0,0,0;1)&\text{n/a}&\text{n/a}&\text{fibration}&\text{conic bundle over }\P^3\\
\hline
4&(1,1,1;0)&\cong&\text{n/a}&\text{fibration}&dP_3\slash\P^2\\
\hline
5&(0,1,1;0)&3\times(1,1,1,-1,-1)\text{ flips}&\text{n/a}&\text{fibration}&(Y_3\subset\mathbb{P}^4)\slash\P^1 \\
\hline
6&(0,1,2;0)&3\times(1,1,1,-1,-2)\text{ flops}&\text{n/a}&\text{contraction}&Y_6\subset\P(1,1,1,1,2,2)\\
\hline
7&(0,0,1;0)&\text{n/a}&\text{n/a}&\text{contraction}&\text{Fano }Y_3\subset\P^5\\
\hline
8&(0,0,2;0)&\text{n/a}&\text{n/a}&\text{contraction}&\text{Fano }Y_6\subset\P(1,1,1,2,2,2)\\
\hline
9&(0,2,2;0)&3\times(1,1,1,-2,-2)\text{ antiflip}&\text{n/a}&\text{fibration}&(Y_6\subset\P(1^3,2^2))\slash\P^1\\ 
\hline
10&(1,1,1,-1)&(1,1,1,-1,-1)\text{ flip}&\text{n/a}&\text{fibration}&dP_8\slash\P^2\\
\hline
11&(1,1,2;-1)&(1,1,1,-1,-2)\text{ flop}&\text{n/a}&\text{contraction}&Y_5\subset\P(1^5,2)\\
\hline
12&(1,1,2;-2)&(1,1,1,-1,-1)\text{ flip}&\text{n/a}&\text{contraction}&Y_4\subset\mathbb{P}(1^5,2)\\
\hline
13&(1,1,3;-2)&(1,1,1,-1,-1,-3;-2)\text{ flop}&\text{n/a}&\text{contraction}&Y_7\subset\P(1^3,2^2,3)\\
\hline
14&(1,1,3;-3)&(1,1,1,-1,-1)\text{ flip}&\text{n/a}&\text{contraction}&Y_5\subset\P(1^3,2^2,3)\\
\hline
15&(1,2,2;-1)&(1,1,1,-2,-2)\text{ antiflip}&(1,1,1,1,-2,-2;2)\text{ flop}&\text{fibration}&(Y_5\subset\P(1^4,2))\slash\P^1\\
\hline
16&(1,2,2;-2)&(1,1,1,-2,-2)\text{ flop}&(1,1,1,-1,-1)\text{ flip}&\text{fibration}&(Y_4\subset\P(1^4,2))\slash\P^1\\
\hline
17&(1,1,4;-3)&(1,1,1,-1,-1,-4;-3)\text{ flop}&\text{n/a}&\text{contraction}&Y_{10}\subset\P(1^3,3^2,4)\\
\hline
18&(1,2,3;-3)&(1,1,1,-1,-3)\text{ antiflip}&(1,1,1,-1,-2)\text{ flop}&\text{contraction}&Y_7\subset\P(1^4,2,3)\\
\hline
19&(1,2,3;-3)&(1,1,1,-1,-2)\text{ flop}&\cong&\text{contraction}&Y_6\subset\P(1^4,2,3)
\end{array}\]
\caption{{\footnotesize Part 1 data of Type~\III\ and \IV\ links from general degree~3 del Pezzo hypersurface fibrations over $\P^2$\label{table!dP3list1}}}
\end{table}

\begin{table}[ht]
\vspace{0.7cm}
\[\begin{array}{cc||c|c|c|c}
\text{No.}&(a,b,c;d)&\psi_1&\psi_2&\varphi^\prime&\text{new model}\\
\hline
\hline
20&(2,2,2;-2)&(1,1,1,-2,-2)\text{ antiflip}&\text{n/a}&\text{fibration}&dP_2\slash\P^2\\
\hline
21&(1,2,4;-3)&(1,1,1,-1,-3)\text{ antiflip}&\cong&\text{contraction}&Y_9\subset\P(1^3,2,3,4)\\
\hline
22&(1,3,3;-3)&(1,1,1,-1,-3)\text{ antiflip}&\cong&\text{fibration}&(Y_6\subset\P(1^3,2,3))\slash\P^1\\
\hline
23&(2,2,3;-3)&(1,1,1,-2,-2)\text{ antiflip}&\text{n/a}&\text{contraction}&Y_6\subset\P(1^5,3)\\
\hline
24&(1,3,4;-3)&(1,1,1,-2,-2)\text{ antiflip}&\cong&\text{contraction}&Y_9\subset\P(1^4,3,4)\\
\hline
25&(2,2,4;-4)&(1,1,1,-2,-2)\text{ antiflip}&\text{n/a}&\text{contraction}&Y_8\subset\P(1^3,2^2,4)\\
\hline
26&(2,3,3;-3)&(1,1,1,-2,-3)\text{ antiflip}&(1,1,1,2,-1,-1;3)\text{ flop}&\text{fibration}&(Y_6\subset\P(1^4,3))\slash\P^1\\
\hline
27&(1,4,4;-3)&(1,1,1,-1,-4,-4;-3)\text{ antiflip}&\cong&\text{fibration}&(Y_9\subset\P(1^3,3,4))\slash\P^1\\
\hline
28&(2,2,5;-5)&(1,1,1,-2,-5)\text{ ntiflip}&\text{n/a}&\text{contraction}&Y_{10}\subset\P(1^3,3^2,5)\\
\hline
29&(2,3,4;-4)&(1,1,1,-2,-3)\text{ antiflip}&(1,1,1,-1,-2)\text{ flop}&\text{contraction}&Y_8\subset\P(1^4,2,4)\\
\hline
30&(2,3,5,-5)&(1,1,1,-2,-3)\text{ antiflip}&(1,1,2,-1,-3)\text{ flop}&\text{contraction}&Y_{10}\subset\P(1^3,2,3,5)\\
\hline
31&(2,4,4;-4)&(1,1,1,-2,-4)\text{ antiflip}&(1,1,1,1,-2,-2;2)\text{ antiflip}&\text{fibration}&(Y_8\subset\P(1^3,2,4))\slash\P^1\\
\hline
32&(2,3,6;-6)&(1,1,1,-2,-3)\text{ antiflip}&\cong&\text{contraction}&Y_{12}\subset\P(1^3,3,4,6)\\
\hline
33&(2,4,5;-5)&(1,1,1,-2,-4)\text{ antiflip}&(1,1,2,-2,-3)\text{ antiflip}&\text{contraction}&Y_{10}\subset\P(1^4,3,5)\\
\hline
34&(2,4,6;-6)&(1,1,1,-2,-4)\text{ antiflip}&\cong&\text{contraction}&Y_{12}\subset\P(1^3,2,4,6)\\
\hline
35&(2,5,5;-5)&(1,1,1,-2,-5)\text{ antiflip}&(1,1,2,-3,-3)\text{ antiflip}&\text{fibration}&(Y_{10}\subset\P(1^3,3,5))\slash\P^1\\
\hline
36&(2,5,6;-6)&(1,1,1,-2,-5)\text{ antiflip}&\cong&\text{contraction}&Y_{12}\subset\P(1^4,4,6)\\
\hline
37&(2,6,6;-6)&(1,1,1,-2,-6)\text{ antiflip}&\cong&\text{fibration}&(Y_{12}\subset\P(1^3,4,6))\slash\P^1
\end{array}\]
\caption{{\footnotesize Part 2 data of Type~\III\ and \IV\ links from general degree~3 del Pezzo hypersurface fibrations over $\P^2$\label{table!dP3list2}}}
\end{table}\end{landscape}

\bibliographystyle{amsplain}
\bibliography{bib}

\vspace{0.4cm}

\noindent Universit\"{a}t Basel, Mathematisches Institut, Rheinsprung 21,  CH-4051 Basel, Switzerland \\       
{\it E-mail:}  \url{hamid.ahmadinezhad@unibas.ch}

\end{document}